\documentclass{amsart}
\usepackage{graphicx} % Required for inserting images
\usepackage{amssymb}
\usepackage{amsmath}
\usepackage{amsthm}
\usepackage{mathrsfs}
\usepackage{cite}
\usepackage{hyperref}
\usepackage[all]{xy}
\usepackage{tikz}
\usetikzlibrary{patterns}
\usepackage{tabularx}
\usepackage{pgfplots}
\pgfplotsset{compat=1.18}

\newtheorem{theorem}{\textbf{Theorem}}[section]

\newtheorem{proposition}[theorem]{\textbf{Proposition}}

\newtheorem{lemma}[theorem]{\textbf{Lemma}}

\hypersetup{
	colorlinks=true,
	linkcolor=blue,
	citecolor=blue
}

\numberwithin{equation}{section}

\title{Improvements of Classical Eigenvalue Inequalities I:  Quantitative Hersch-deficit  and the Joint Spectral Range on the Two-Sphere}

\date{\today}

\author{Zuoqin Wang and Qingwei Zeng}
\thanks{Partially supported by NNSFC No. 12571064.} 

\address[Z. W.]{School of Mathematical Sciences\\
	University of Science and Technology of China\\
	Hefei, 230026\\ P.R. China} 
\email{wangzuoq@ustc.edu.cn}

\address[Q. Z.]{School of Mathematical Sciences\\
	University of Science and Technology of China\\
	Hefei, 230026\\ P.R. China} 
\email{qwzeng@mail.ustc.edu.cn}

\begin{document}
	
\maketitle

\markboth{ZUOQIN WANG AND QINGWEI ZENG}{QUANTITATIVE IMPROVEMENTS OF CLASSICAL EIGENVALUE INEQUALITIES}
	
	\begin{abstract}
		For area-$4\pi$ metrics on the two-sphere, we obtain quantitative refinements of classical eigenvalue inequalities and study the joint behavior of the first two Laplace eigenvalues. We prove a sharp quadratic improvement of Hersch's reciprocal-sum inequality in terms of the first eigenvalue, together with a simultaneous quadratic estimate for the first three eigenvalues and a geometric stability result for admissible measures. We also establish a quantitative refinement of Nadirashvili's second-eigenvalue inequality with sharp exponential scale. Finally, using explicit spectral families due to Petrides, a degree argument, and a refined finite-dimensional matrix method, we obtain nontrivial inner and outer bounds for the joint range of $(\lambda_1,\lambda_2)$.  
	    %The proofs combine finite-dimensional compression of the inverse Laplacian onto the first spherical harmonics, exact spherical moment identities, perturbative arguments, and capacity estimates near the two-bubble degeneration.
\end{abstract} 

\section{Introduction}

Let $(M,g)$ be a closed connected Riemannian surface, and let
$\Delta_g=-\operatorname{div}_g\circ\nabla$ be the Laplace--Beltrami
operator. Its spectrum, counted with multiplicity, is
\begin{equation*}
	0=\lambda_0(M,g)<\lambda_1(M,g)\leq\lambda_2(M,g)\leq\cdots\nearrow\infty.
\end{equation*} 
A basic problem in spectral geometry is to understand the spectra that can
arise as the metric varies. Since the full spectral range is still poorly
understood, one typically imposes geometric constraints and seeks sharp
inequalities for finitely many eigenvalues.

For surfaces, the most natural normalization is the area. Indeed, under a
homothetic rescaling one has
\begin{equation*}
	\lambda_k(M,tg)=t^{-1}\lambda_k(M,g),\qquad
	\operatorname{Area}(M,tg)=t\operatorname{Area}(M,g),
\end{equation*}
so the quantities $\lambda_k(M,g)\operatorname{Area}(M,g)$ are scale invariant. 
Under a fixed-area constraint there is, in general, no positive lower bound
for a prescribed eigenvalue $\lambda_k$, as is seen from standard dumbbell
degenerations. The natural questions therefore concern sharp upper bounds,
rigidity and stability of extremizers, and relations among several low-lying
eigenvalues. General topological upper bounds were obtained by Yang--Yau for
the first eigenvalue and by Korevaar for higher eigenvalues
\cite{yang-yau,kor}; sharp results are known in several special settings, in
particular for the sphere and the real projective plane
\cite{hersch,li-yau,nad,klein,nay-sho,knpp,kar_real_projective_spaces}.

This is the first paper in a series in which we aim to sharpen classical
eigenvalue inequalities in a way that retains information about the joint
behavior of low-lying eigenvalues. A recurring theme is that a small deficit
in a reciprocal-sum inequality forces several eigenvalues to be simultaneously
close to their extremal values. In the present paper we study the first three
eigenvalues on the two-sphere of area $4\pi$, as well as the image of the joint
spectral map determined by the first two eigenvalues.

Let $g_{\rm rd}$ denote the round metric of area $4\pi$ on $\mathbb S^2$.
Hersch \cite{hersch} proved that every Riemannian metric $g$ on $\mathbb S^2$
of area $4\pi$ satisfies
\begin{equation}\label{Hersch}
	\frac{1}{\lambda_1(\mathbb S^2,g)}
	+\frac{1}{\lambda_2(\mathbb S^2,g)}
	+\frac{1}{\lambda_3(\mathbb S^2,g)}
	\geq \frac{3}{2},
\end{equation} 
with equality if and only if $g$ is isometric to the round metric. Since
$\lambda_1\leq\lambda_2\leq\lambda_3$, \eqref{Hersch} immediately yields the
sharp bound $\lambda_1(\mathbb S^2,g)\leq2$. A second classical result is
Nadirashvili's sharp estimate \cite{nad_second_eigenvalue_of_spehre}: if
$\operatorname{Area}(\mathbb S^2,g)=4\pi$, then 
\begin{equation}\label{second eigenvalue}
	\lambda_2(\mathbb S^2,g)<4.
\end{equation} 
The value $4$ is not attained by a smooth metric; it is approached by metrics
degenerating to a wedge of two identical round spheres of area $2\pi$. The
higher-eigenvalue picture and the structure of maximizing sequences are
discussed, for example, in \cite{knpp}. Our purpose is to quantify both of
these extremal phenomena while retaining enough information to study the joint
behavior of $\lambda_1$ and $\lambda_2$. 

It is convenient to work more generally with eigenvalues of measures in the
fixed conformal class of the round sphere. The precise definition of an
admissible measure and the associated min--max eigenvalues is recalled in
Section~\ref{sec:prelim-sphere}. Smooth area measures of conformal metrics are
included in this framework, while the measure formulation also accommodates
the concentration phenomena arising near the Nadirashvili extremal
configuration; see Kokarev \cite{kok_eigenvalues_of_measures}.  

Reciprocal sums have a longer variational history than the individual
isoperimetric inequalities. Hersch formulated a trace-type variational
principle for sums of consecutive reciprocal eigenvalues
\cite{hersch_variational}, and related reciprocal-sum estimates were developed
by Hile--Xu \cite{hilexu} and Dittmar \cite{dittmar}. More recently, Eddaoudi
\cite{eddaoudi} obtained reciprocal-sum lower bounds for spheres and closed
orientable surfaces. In a different geometric setting, He--Li--Tang
\cite{hlt} proved the Ashbaugh--Benguria conjecture for the first $d$ nonzero
Neumann eigenvalues of Euclidean domains by keeping the entire transplanted
$d$-dimensional trial space coupled and comparing its mass and energy matrices.
Li--Wang \cite{liwang} developed a quantitative version of that matrix
mechanism and, in dimension two, used the remaining matrix information to
obtain additional constraints on the joint first-two-eigenvalue image.
Our method is close in spirit to the works described above. In both settings, the key point is to keep a finite-dimensional family of trial functions coupled, rather than estimating the coordinate functions separately, and to retain the matrix information that records the directional deviation from the symmetric configuration. In our setting, this deviation is encoded by the trace-free part $A=B-\frac12 I$ of the matrix $B$ describing the action of the inverse Laplace operator on the space of degree-one spherical harmonics. The mechanism that controls this matrix, however, is different. On the sphere, the crucial estimates come from exact fourth- and sixth-order moment identities for the coordinate functions on the round sphere. In a series of works currently in preparation, we further develop this method for general surfaces and higher-dimensional manifolds, obtaining new inequalities for sums of reciprocals of eigenvalues together with corresponding quantitative stability results.

\subsection{Improvements of Hersch's inequality}

For an admissible measure $\mu$ of mass $4\pi$ on $\mathbb S^2$, we study the Hersch deficit 
\begin{equation}\label{eq:Dk-sphere-intro}
	\mathcal H(\mu):=\sum_{i=1}^3\frac1{\lambda_i(\mu)}-\frac{3}{2}.
\end{equation}
For smooth metrics $g$, we use the notation $\mathcal H(g)=\mathcal H(dv_g)$.

Our first result gives the sharp quadratic remainder of Hersch's inequality \eqref{Hersch} in terms of the first
eigenvalue alone.
\begin{theorem}\label{main_thm_sphere_first}
	Let $g$ be a Riemannian metric on $\mathbb S^2$ with
	$\operatorname{Area}(\mathbb S^2,g)=4\pi$. Then
	\begin{equation}\label{improvedHersc_first}
		\mathcal H(g)\geq\frac{9}{160}(\lambda_1(g)-2)^2.
	\end{equation}
	Moreover, both the exponent $2$ and the constant $\frac{9}{160}$ are sharp.
\end{theorem}

We next obtain a global remainder that simultaneously controls the deviations
of the first three eigenvalues from the round value.
\begin{theorem}\label{main_thm_sphere}
	Let $g$ be a Riemannian metric on $\mathbb S^2$ with
	$\operatorname{Area}(\mathbb S^2,g)=4\pi$. Then
	\begin{equation}\label{improvedHersc}
		\mathcal H(g)\geq\frac{1}{183}\sum_{i=1}^3(\lambda_i(g)-2)^2.
	\end{equation}
	Moreover, the exponent $2$ is sharp.
\end{theorem}
In contrast to Theorem \ref{main_thm_sphere_first}, the coefficient  $\frac{1}{183}$ here is not sharp. 

The Hersch deficit also controls the distance of the measure from the
conformal orbit of the round measure, giving a quantitative geometric form of
Hersch rigidity and extending the stability perspective of
\cite[Theorem 1]{knps}. 
\begin{theorem}\label{thm:sharp-reciprocal-hierarchy}
	Let $\mu$ be an admissible measure on $(\mathbb S^2,[g_{\rm rd}])$ with
	$\mu(\mathbb S^2)=4\pi$ and $\mathcal H(\mu)\leq\frac12$. Then there exists a
	conformal automorphism $\Phi$ such that
	\begin{equation}\label{eq:hierarchy-geo-intro}
		\mathcal H(\mu)\geq\frac{1}{344}
		\|\Phi_*\mu-dv_{g_{\rm rd}}\|_{W^{-1,2}}^2.
	\end{equation}
	The geometric exponent $2$ is sharp.
\end{theorem}

Here $W^{-1,2}$ stands for the dual norm of Sobolev norm $W^{1,2}(\mathbb{S}^2,g_{\rm rd})$. Since $\mathcal{H}(\mu)\leq3(\frac{1}{\lambda_1}-2)$, the sharpness of exponent $2$ here follows the sharpness of \cite[Theorem 1]{knps}. 

\subsection{Distribution of the first two eigenvalues}

Next we give a quantitative improvement of Nadirashvili's inequality \eqref{second eigenvalue}. 
Near the Nadirashvili extremal regime, where the conformally normalized measures converge to
$\mu_0=\frac12dv_{g_{\rm rd}}+2\pi\delta_p$, one has simultaneously
$\lambda_2\to4$ and $\lambda_1\to0$. Our next result quantifies the relation between these two asymptotic behaviors and identifies the sharp exponential scale governing the approach of $\lambda_2$ to $4$.

\begin{theorem}\label{main_stability_second}
	For every admissible measure $\mu$ on $(\mathbb S^2,[g_{\rm rd}])$ with
	$\mu(\mathbb S^2)=4\pi$,
	\begin{equation}\label{improvedNadir}
		\lambda_2(\mu)<4-\frac{e}{25}
		\exp\left(-\frac{2}{\lambda_1(\mu)}\right).
	\end{equation}
	Moreover, the constant $2$ in the exponential is sharp.
\end{theorem}

The second main purpose of this paper is to study the joint spectral range 
\begin{equation*}
	\mathscr R:=\{(\lambda_1(\mu),\lambda_2(\mu)):
	\mu\text{ admissible},\ \mu(\mathbb S^2)=4\pi\}.
\end{equation*} 
Petrides constructed in \cite{pet25} a family of Riemannian metrics on $\mathbb S^2$ with area $4\pi$ that connects $(0,4)$ to $(2,2)$. 
In Proposition~\ref{measures of multiplicity 2} we employ irreducible representations of the symmetry group $S_3$ to construct a family of 
admissible measures of mass $4\pi$ that realize the point $(\lambda,\lambda)$, $\lambda \in (0,2]$. %Thus the diagonal segment ending at the round point $(2,2)$ is contained in the joint spectral range (with the origin occurring only as a limiting point).
Using a  convex interpolation of measures and a degree argument, we are able to obtain in Section~\ref{sec:joint-range}  an
explicit inner region $\Omega$ in $\mathscr R$. On the other hand, the inequalities above and the refined matrix estimate of Section~\ref{sec:first-two-refined} provide
outer restrictions. 

Figure~\ref{fig:joint-region-intro} summarizes these two sides of the picture.
The region $\Omega$ is proved in Proposition~\ref{imageunderred} to lie in the
joint spectral range. The region $\widetilde\Omega$ between the constructed
inner family and the available outer bounds is not decided by the present
results. The pink hatched region is rigorously excluded from the joint spectral
range by Proposition~\ref{prop:finite-scale-capacity} and
Theorem~\ref{thm:uniform-directional-bound}. The thick blue curve represents
Theorem~\ref{thm:uniform-directional-bound}; the green curve is obtained by
plotting the numerical values in Table~\ref{Numerical result of improved nad},
which come from Proposition~\ref{prop:finite-scale-capacity}; and the red curve
is the Petrides family of
Section~\ref{subsec:petrides-double-bubble}. The thick purple diagonal is the
family furnished by Proposition~\ref{measures of multiplicity 2}, while the
gray dashed horizontal line is Nadirashvili's bound $\lambda_2<4$. 

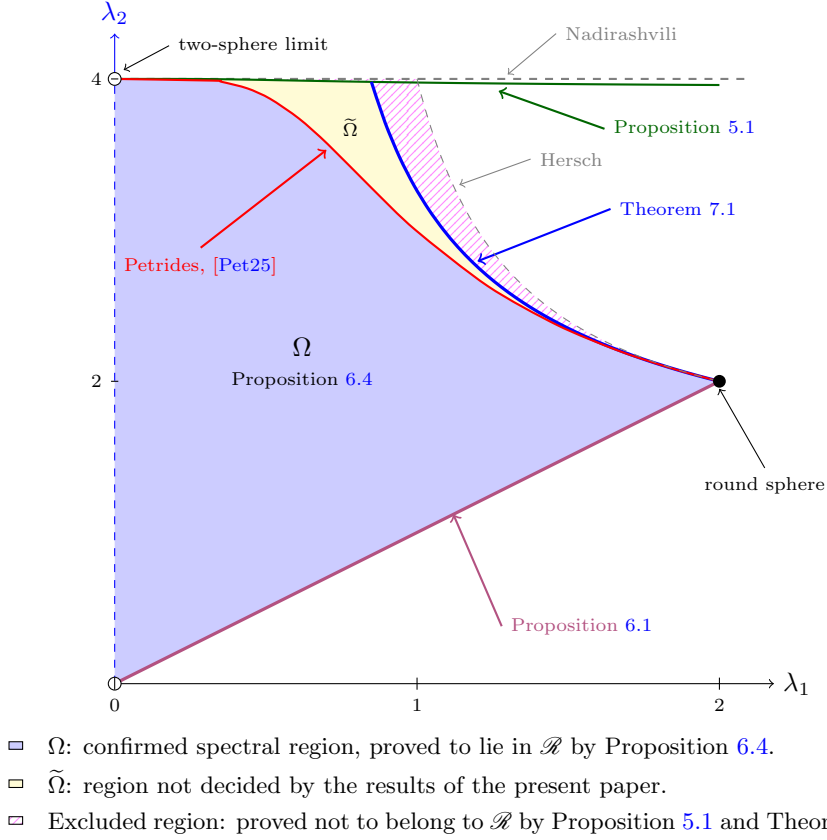
\begin{figure}[ht]
	\centering
	\begin{tikzpicture}[xscale=4, yscale=2]
		% Confirmed spectral region: Proposition 6.4.
		\fill[blue!20]
		(0,0) -- (2,2) --
		plot[smooth] coordinates {
			(2.000000,2.000000) (1.923696,2.039744) (1.817110,2.101145)
			(1.661246,2.205839) (1.462301,2.372995) (1.233750,2.629891)
			(0.992418,3.004651) (0.847537,3.282707) (0.704537,3.568889)
			(0.609812,3.737477) (0.513863,3.868945) (0.438650,3.937623)
			(0.355283,3.980276) (0.323300,3.989064) (0.000000,4.000000)} -- cycle;
		\node[black,align=center] at (0.62,2.12)
		{$\Omega$\\[-1pt]{\scriptsize Proposition~\ref{imageunderred}}};
		
		% Region not decided by the present results.  The outer boundary follows
		% Proposition 5.1 up to its intersection with Theorem 7.1, and then Theorem 7.1.
		\fill[yellow!20]
		(0.000000,4.000000) --
		plot[smooth] coordinates {
			(0.323300,3.989064) (0.355283,3.980276) (0.438650,3.937623)
			(0.513863,3.868945) (0.609812,3.737477) (0.704537,3.568889)
			(0.847537,3.282707) (0.992418,3.004651) (1.233750,2.629891)
			(1.462301,2.372995) (1.661246,2.205839) (1.817110,2.101145)
			(1.923696,2.039744) (2.000000,2.000000)} --
		plot[smooth] coordinates {
			(1.951921,2.024708) (1.903841,2.050834) (1.855762,2.078508)
			(1.807683,2.107880) (1.759603,2.139125) (1.711524,2.172443)
			(1.663445,2.208067) (1.615365,2.246271) (1.567286,2.287378)
			(1.519207,2.331773) (1.471127,2.379916) (1.423048,2.432361)
			(1.374969,2.489783) (1.326890,2.553014) (1.278810,2.623083)
			(1.230731,2.701287) (1.182652,2.789277) (1.134572,2.889190)
			(1.086493,3.003839) (1.038414,3.137009) (0.990334,3.293905)
			(0.942255,3.481898) (0.894176,3.711779) (0.848997,3.980538)} --
		plot[smooth] coordinates {
			(0.848997,3.980539) (0.800000,3.98211690) (0.750000,3.98381051)
			(0.700000,3.98558425) (0.650000,3.98742978) (0.600000,3.98933134)
			(0.550000,3.99126220) (0.500000,3.99318020) (0.450000,3.99502263)
			(0.400000,3.99670247) (0.350000,3.99811136) (0.300000,3.99914061)
			(0.250000,3.99973407) (0.200000,3.99995892) (0.150000,3.99999842)
			(0.100000,4.00000000) (0.050000,4.00000000) (0.000000,4.00000000)} -- cycle;
		\node[black] at (0.78,3.68) {\scriptsize $\widetilde\Omega$};
		
		% Region rigorously excluded by Proposition 5.1 and Theorem 7.1.
		% Pink hatching replaces the previous solid pale-red fill.
		\path[pattern=north east lines,pattern color=magenta!45]
		(0.000000,4.000000) -- (1,4) --
		plot[samples=80,domain=1:2] (\x,{2/(1.5-1/\x)}) --
		plot[smooth] coordinates {
			(1.951921,2.024708) (1.903841,2.050834) (1.855762,2.078508)
			(1.807683,2.107880) (1.759603,2.139125) (1.711524,2.172443)
			(1.663445,2.208067) (1.615365,2.246271) (1.567286,2.287378)
			(1.519207,2.331773) (1.471127,2.379916) (1.423048,2.432361)
			(1.374969,2.489783) (1.326890,2.553014) (1.278810,2.623083)
			(1.230731,2.701287) (1.182652,2.789277) (1.134572,2.889190)
			(1.086493,3.003839) (1.038414,3.137009) (0.990334,3.293905)
			(0.942255,3.481898) (0.894176,3.711779) (0.848997,3.980538)} --
		plot[smooth] coordinates {
			(0.848997,3.980539) (0.800000,3.98211690) (0.750000,3.98381051)
			(0.700000,3.98558425) (0.650000,3.98742978) (0.600000,3.98933134)
			(0.550000,3.99126220) (0.500000,3.99318020) (0.450000,3.99502263)
			(0.400000,3.99670247) (0.350000,3.99811136) (0.300000,3.99914061)
			(0.250000,3.99973407) (0.200000,3.99995892) (0.150000,3.99999842)
			(0.100000,4.00000000) (0.050000,4.00000000) (0.000000,4.00000000)} -- cycle;
		
		\draw[->] (0,0) -- (2.18,0) node[right] {$\lambda_1$};
		\draw[dashed,blue] (0,0) -- (0,4);
		\draw[->,blue] (0,4) -- (0,4.30) node[above] {$\lambda_2$};
		
		% Nadirashvili horizontal bound, now with an explicit pointer.
		\draw[thick,gray,dashed] (0,4) -- (2.10,4);
		\node[gray,anchor=west] (nadlabel) at (1.46,4.3)
		{\scriptsize Nadirashvili};
		\draw[->,gray] (nadlabel.west) -- (1.3,4.03);
		
		% Nontrivial diagonal family: Proposition 6.1.
		\draw[very thick,magenta!70!black] (0,0) -- (2,2);
		\node[magenta!70!black,align=left,anchor=west] (diaglabel) at (1.28,0.38)
		{\scriptsize  
			 Proposition~\ref{measures of multiplicity 2}};
		\draw[->,magenta!70!black,thick] (diaglabel.west) -- (1.12,1.12);
		
		% Proposition 5.1: numerical curve obtained by plotting Table 1.
		\draw[thick,green!40!black] plot[smooth] coordinates {
			(0.000000,4.00000000) (0.050000,4.00000000) (0.100000,4.00000000)
			(0.150000,3.99999842) (0.200000,3.99995892) (0.250000,3.99973407)
			(0.300000,3.99914061) (0.350000,3.99811136) (0.400000,3.99670247)
			(0.450000,3.99502263) (0.500000,3.99318020) (0.550000,3.99126220)
			(0.600000,3.98933134) (0.650000,3.98742978) (0.700000,3.98558425)
			(0.750000,3.98381051) (0.800000,3.98211690) (0.850000,3.98050677)
			(0.900000,3.97898027) (0.950000,3.97753549) (1.000000,3.97616933)
			(1.050000,3.97487798) (1.100000,3.97365732) (1.150000,3.97250310)
			(1.200000,3.97141114) (1.250000,3.97037739) (1.300000,3.96939799)
			(1.350000,3.96846929) (1.400000,3.96758790) (1.450000,3.96675064)
			(1.500000,3.96595456) (1.550000,3.96519692) (1.600000,3.96447520)
			(1.650000,3.96378706) (1.700000,3.96313035) (1.750000,3.96250307)
			(1.800000,3.96190338) (1.850000,3.96132959) (1.900000,3.96078010)
			(1.950000,3.96025348) (2.000000,3.95974838)};
		\node[green!40!black,anchor=west] (greenlabel) at (1.62,3.67)
		{\scriptsize Proposition~\ref{prop:finite-scale-capacity}};
		\draw[->,green!40!black,thick] (greenlabel.west) -- (1.27,3.92);
		
		% Theorem 7.1.
		\draw[very thick,blue] plot[smooth] coordinates {
			(0.848997,3.980538) (0.894176,3.711779) (0.942255,3.481898)
			(0.990334,3.293905) (1.038414,3.137009) (1.086493,3.003839)
			(1.134572,2.889190) (1.182652,2.789277) (1.230731,2.701287)
			(1.278810,2.623083) (1.326890,2.553014) (1.374969,2.489783)
			(1.423048,2.432361) (1.471127,2.379916) (1.519207,2.331773)
			(1.567286,2.287378) (1.615365,2.246271) (1.663445,2.208067)
			(1.711524,2.172443) (1.759603,2.139125) (1.807683,2.107880)
			(1.855762,2.078508) (1.903841,2.050834) (1.951921,2.024708)
			(2.000000,2.000000)};
	%	\filldraw[fill=white,draw=blue] (0.848997,3.980538) circle (0.9pt);
		\node[blue,anchor=west] (bluelabel) at (1.64,3.14)
		{\scriptsize Theorem~\ref{thm:uniform-directional-bound}};
		\draw[->,blue,thick] (bluelabel.west) -- (1.2,2.794);
		
		% Explicit double-bubble family from Section 5.2.
		\draw[red,thick] plot[smooth] coordinates {
			(2.000000,2.000000) (1.923696,2.039744) (1.817110,2.101145)
			(1.661246,2.205839) (1.462301,2.372995) (1.233750,2.629891)
			(0.992418,3.004651) (0.847537,3.282707) (0.704537,3.568889)
			(0.609812,3.737477) (0.513863,3.868945) (0.438650,3.937623)
			(0.355283,3.980276) (0.323300,3.989064) (0.000000,4.000000)};
		\node[red,align=left,anchor=west] (dblabel) at (0,2.76)
		{\scriptsize Petrides, \cite{pet25}};
		\draw[->,red,thick] (dblabel.north) -- (0.7,3.53);
		
		% Hersch reciprocal-sum hyperbola, with pointer.
		\draw[dashed,gray,samples=80,domain=1:2]
		plot (\x,{2/(1.5-1/\x)});
		\node[gray] (herschlabel) at (1.51,3.47) {\scriptsize Hersch};
		\draw[->,gray] (herschlabel.west) -- (1.14,3.3);
		
		% Distinguished points, with explicit pointer arrows.
		\fill (2,2)  ellipse [x radius=0.6pt, y radius=1.2pt];%circle (0.8pt);
		\node[anchor=east] (roundlabel) at (2.38,1.31) {\scriptsize round sphere};
		\draw[->,black] (roundlabel.north) -- (2,1.95);
		
		\filldraw[fill=white,draw=black] (0,4)  ellipse [x radius=0.6pt, y radius=1.2pt];
		\node[anchor=west] (twolabel) at (0.18,4.22) {\scriptsize two-sphere limit};
		\draw[->,black] (twolabel.west) -- (0.03,4.06);
		
		\filldraw[fill=white,draw=black] (0,0)  ellipse [x radius=0.6pt, y radius=1.2pt];
		
		\foreach \x in {0,1,2} {\draw (\x,0.04)--(\x,-0.04) node[below] {\scriptsize \x};}
		\foreach \y in {2,4} {\draw (0.012,\y)--(-0.012,\y) node[left] {\scriptsize \y};}
	\end{tikzpicture}
	
	\small
	\begin{minipage}{0.94\linewidth}
		\vspace{2mm}
		\noindent
		\tikz[baseline=-0.5ex]{\fill[blue!20] (0,0) rectangle (0.18,0.10);\draw (0,0) rectangle (0.18,0.10);}\quad
		$\Omega$: confirmed spectral region, proved to lie in $\mathscr R$ by Proposition~\ref{imageunderred}.\\[1mm]
		\tikz[baseline=-0.5ex]{\fill[yellow!20] (0,0) rectangle (0.18,0.10);\draw (0,0) rectangle (0.18,0.10);}\quad
		$\widetilde\Omega$: region not decided by the results of the present paper.\\[1mm]
		\tikz[baseline=-0.5ex]{\path[pattern=north east lines,pattern color=magenta!45] (0,0) rectangle (0.18,0.10);\draw (0,0) rectangle (0.18,0.10);}\quad
		Excluded region: proved not to belong to $\mathscr R$ by Proposition~\ref{prop:finite-scale-capacity} and Theorem~\ref{thm:uniform-directional-bound}.
	\end{minipage}
	\caption{Joint first-two-eigenvalue picture.}
	\label{fig:joint-region-intro}
\end{figure}

\subsection{Organization of the paper}

Section~\ref{sec:prelim-sphere} develops the analytic framework for admissible
measures, the inverse operator, and the spherical moment identities used in
the matrix estimates. Section~{\color{blue}\ref{sec:sharp-first-remainder}} proves the sharp first-eigenvalue
remainder in Theorem~\ref{main_thm_sphere_first}. 
Section~\ref{sec:quadratic-expanded} proves the global
three-eigenvalue remainder in Theorem~\ref{main_thm_sphere} and the geometric
stability statement in Theorem~\ref{thm:sharp-reciprocal-hierarchy}.
 Section~\ref{proof of second eigenvalue} develops the quantitative
Nadirashvili estimate and the explicit double-bubble family, leading to Theorem~\ref{main_stability_second}. 
Section~\ref{sec:joint-range} first proves the nontrivial diagonal-realization
result, Proposition~\ref{measures of multiplicity 2}, showing that every point 
$(\lambda,\lambda)$ with $0<\lambda\le2$ belongs to the joint spectral range;
it then constructs a larger inner region by convex interpolation and a degree
argument. Finally, Section~\ref{sec:first-two-refined} proves the uniform
first-two-eigenvalue bound in Theorem~\ref{thm:uniform-directional-bound} and
analyzes its explicit boundary function.

\section{Preliminaries}\label{sec:prelim-sphere}

%This section fixes the analytic framework and collects all matrix identities used later.  Keeping these ingredients in one place avoids repeating the construction of the inverse operator and the spherical moment computations in several proofs.

\subsection{Admissible measures and the inverse operator}
\label{subsec:admissible-K}

Throughout, let $g_{\text{rd}}$ be the round metric on $\mathbb{S}^2$ with area $4\pi$.  A Radon measure $\mu$ on $\mathbb{S}^2$ is called \emph{admissible} if the identity map on $C^\infty(\mathbb S^2)$ extends to a compact operator $W^{1,2}(\mathbb{S}^2,g_{\rm rd})\rightarrow L^2(\mu)$. For such a measure the generalized Laplace eigenvalues are (see
\cite{kok_eigenvalues_of_measures,knps})
\begin{equation}\label{def_of_eigenvalue}
    \lambda_k(\mu)=\inf_{V_{k+1}}    \sup_{0\neq f\in V_{k+1}}
    \frac{\displaystyle \int_{\mathbb{S}^2}|df|_{g_{\rm rd}}^2\,dv_{g_{\rm rd}}}
         {\displaystyle\int_{\mathbb{S}^2}f^2\,d\mu},
    \qquad k\geq 0,
\end{equation}
where the infimum is taken over all $k+1$ dimensional subspace $V_{k+1} \subset C^\infty(\mathbb S^2)$ that are also $(k+1)$-dimensional in $L^2(\mu)$. %The constants give $\lambda_0(\mu)=0$.  
According to \cite{GKL}, one has 
\[0=\lambda_0(\mu) < \lambda_1(\mu) \le \lambda_2(\mu) \le \cdots \nearrow \infty,\]
Moreover, there exists an orthogonal basis of eigenfunctions $\phi_j \in W^{1,2}(\mathbb{S}^2,\mu)$, 
where $W^{1,2}(\mathbb{S}^2,\mu)$ refers to the completion of $C^{\infty}(\mathbb{S}^2)$ under the norm 
\[ \|u\|^2_{W^{1,2}(\mathbb{S}^2,\mu)}=\|u\|^2_{L^2(\mu)}+\|du\|^2_{L^2(g_{\text{rd}})}.\]

If $c>0$, then directly from the
Rayleigh quotient
\begin{equation}\label{eq:measure-scaling}
 	\lambda_k(c\mu)=c^{-1}\lambda_k(\mu).
\end{equation}
If $g=e^{2\omega}g_{\text{rd}}$ is a smooth conformal metric, then
$\lambda_k(dv_g)=\lambda_k(\mathbb{S}^2,g)$ because in dimension two the Dirichlet
energy is conformally invariant.

We shall always use Hersch's balance trick before applying the matrix method.

\begin{lemma}[\cite{hersch}]\label{lem:hersch-normalization}
Let $\mu$ be an admissible measure of mass $4\pi$ on $\mathbb{S}^2$.   There exists a conformal
automorphism $\Phi$ of $\mathbb{S}^2$ such that
\begin{equation}\label{eq:qs-balance}
 \int_{\mathbb{S}^2}x_i\,d(\Phi_*\mu)=0,
 \qquad i=1,2,3.
\end{equation} 
\end{lemma}

By conformal invariance of the two-dimensional
Dirichlet integral, one   gets
\begin{equation}\label{eq:qs-invariance}
	\lambda_k(\Phi_*\mu)=\lambda_k(\mu),\qquad k\ge0.
\end{equation}
From now on, when the matrix method is used, $\mu$ is understood to be
balanced as in \eqref{eq:qs-balance}.  Consider 
\begin{equation*}
    \mathscr H_\mu=\left\{u\in W^{1,2}(\mathbb{S}^2,\mu):\int_{\mathbb{S}^2}u\,d\mu=0\right\},
\end{equation*}
which is the subspace of $W^{1,2}(\mathbb{S}^2,\mu)$ spanned by $\phi_1, \phi_2, \cdots$. 
 We equip $\mathscr H_\mu$ with the Dirichlet inner product
\begin{equation*}
 	(u,v)_D=\int_{\mathbb{S}^2}\langle du,dv\rangle_{g_{\text{rd}}}\,dv_{g_{\text{rd}}}.
\end{equation*}
Then $\mathscr H_\mu$ becomes a Hilbert space. Moreover, by \eqref{def_of_eigenvalue}, on $\mathscr H_\mu$ we have 
\[
\int_{\mathbb S^2} u^2 d\mu \le \frac 1{\lambda_1(\mu)}  	(u,u)_D.
\]
 %the embedding $W^{1,2}(\mathbb{S}^2,\mu) \hookrightarrow L^2(\mu)$ is compact. 
So the Riesz representation theorem define a compact positive
self-adjoint operator
\begin{equation*}
 K_\mu:\mathscr H_\mu\longrightarrow\mathscr H_\mu,
 \qquad
 (K_\mu u,v)_D=\int_{\mathbb{S}^2}uv\,d\mu.
\end{equation*}
By definition it is easy to see that the  eigenvalues of $K_\mu$, listed in nonincreasing order, are $\lambda_j(\mu)^{-1},j\geq 1$, with eigenfunctions $\phi_j$ as well.

\subsection{Spherical moments and the matrix method}
\label{subsec:matrix-prelim}

We first record the moment formula that will be used repeatedly. A convenient reference is Folland \cite{Folland}.

\begin{lemma}\label{lem:momentsonsphere}
	One has the following identities
	\begin{align}
		\int_{\mathbb{S}^2}x_ix_j\,dv_{g_{\text{rd}}}
		&=\frac{4\pi}{3}\delta_{ij},\label{eq:second-moment}\\
		\int_{\mathbb{S}^2}x_ix_jx_kx_\ell\,dv_{g_{\text{rd}}}
		&=\frac{4\pi}{15}
		(\delta_{ij}\delta_{k\ell}+\delta_{ik}\delta_{j\ell}
		+\delta_{i\ell}\delta_{jk}),\label{eq:qs-fourth}\\
		\int_{\mathbb{S}^2}x_{i_1}\cdots x_{i_6}\,dv_{g_{\text{rd}}}
		&=\frac{4\pi}{105}
		\sum_{\text{\rm 15 pairings}}
		\delta_{i_{a_1}i_{b_1}}
		\delta_{i_{a_2}i_{b_2}}
		\delta_{i_{a_3}i_{b_3}}.\label{eq:sixth-moment}
	\end{align}
\end{lemma} 

Let $P$ be the Dirichlet-orthogonal projection from $\mathscr H_\mu$ onto $E=\mathrm{span}\{x_1,x_2,x_3\}$, and set
\begin{equation*}
    B=PK_\mu P\big|_E,
    \qquad
    C=(I-P)K_\mu P:E\longrightarrow E^\perp.
\end{equation*}
In the Dirichlet-orthonormal basis $e_1=\sqrt{\frac{3}{8\pi}}x_1, e_2=\sqrt{\frac{3}{8\pi}}x_2, e_3=\sqrt{\frac{3}{8\pi}}x_3$, one has 
\begin{equation}\label{eq:qs-Bmatrix}
 B_{ij}=\frac{3}{8\pi}\int_{\mathbb{S}^2}x_ix_j\,d\mu,
 \qquad
 \text{tr} B=\frac32.
\end{equation}

Since an orthogonal change in $\mathbb R^3$ preserves Hersch's balance condition \eqref{eq:qs-balance}, by choosing suitable coordinates in $\mathbb R^3$, in what follows we may assume $B$ is diagonal, with $B_{11} \ge B_{22} \ge B_{33}$.  Moreover, by the max--min characterization of eigenvalues, 
 
\begin{equation}\label{eq:qs-interlace}
 B_{ii}\leq\frac{1}{\lambda_i(\mu)},\qquad i=1,2,3.
\end{equation}
Consequently, we have
\begin{equation}\label{Hersch_mu}
    \mathcal{H}(\mu)=\sum_{i=1}^3\left(\frac{1}{\lambda_i(\mu)}-B_{ii}\right)\geq 0.
\end{equation}
This extends Hersch's classical reciprocal-sum inequality \eqref{Hersch} to the setting of admissible measures.

In what follows, $\|\cdot\|_{\mathrm{HS}}$ denotes the Hilbert--Schmidt norm of a matrix.  For the off-diagonal block, we use  $C^*C=PK_\mu^2P-B^2$ to get 
\begin{equation}\label{eq:qs-A-small}
    \begin{split}
        \|C\|_{\text{HS}}^2 =\text{tr}(PK_\mu^2P)-\text{tr}(B^2) \leq\sum_{i=1}^3(\lambda_i^{-2}(\mu)-B_{ii}^2) 
        \leq 
        2\lambda_1^{-1}(\mu)\mathcal H(\mu).
    \end{split}
\end{equation}

Denote $A=B-\frac{1}{2}I$. We prove 

\begin{proposition} \label{cor:mx-A-C} 
	We have $\|A\|_{\text{\rm HS}}\leq 2\|C\|_{\text{HS}}$.
\end{proposition}
\begin{proof} 
By definition and \eqref{eq:qs-fourth},
\begin{equation*}
 \int_{\mathbb S^2}xx^T\,d\mu=\frac{8\pi}{3}B,
 \qquad
 \int_{\mathbb S^2}xx^T(x^TBx)\,dv_{g_{\rm rd}}
 =\frac{4\pi}{15}\left(\frac32I+2B\right).
\end{equation*}
Let
\[
 d\nu=d\mu-2x^TBx\,dv_{g_{\rm rd}}.
\]
Then
\[
 A=\frac{5}{8\pi}\int_{\mathbb S^2}xx^T\,d\nu.
\]
Since $x^Tx=1$, we have the decomposition
\begin{equation}\label{decomposition_nu}
 d\nu=\sum_{i=1}^3x_i^2\,d\mu-2\sum_{i=1}^3x_i(Bx)_i\,dv_{g_{\rm rd}}
 =:\sum_{i=1}^3x_i\,d\eta_i,
\end{equation}
where $d\eta_i=x_i\,d\mu-2(Bx)_i\,dv_{g_{\rm rd}}$. Moreover,
\[
 (Cx_i,f)_D=((I-P)K_\mu Px_i,f)_D
 =(K_\mu x_i-Bx_i,f)_D
 =\int_{\mathbb S^2}f\,d\eta_i.
\]
Hence, for every symmetric trace-free matrix $H$,
\begin{equation*}
 \langle A,H\rangle_{\rm HS}
 =\frac{5}{8\pi}\int_{\mathbb S^2}x^THx\,d\nu
 =\frac{5}{8\pi}\sum_i(Cx_i,x_i(x^THx))_D.
\end{equation*}
Using \eqref{eq:qs-fourth},
\[
 \int_{\mathbb S^2}x_ix_j(x^THx)\,dv_{g_{\rm rd}}
 =\frac{8\pi}{15}H_{ij},
\]
so the projection of $x_i(x^THx)$ onto $E$ is exactly
$\frac25(Hx)_i$. Consequently, since $Cx_i\in E^\perp$,
\begin{equation}\label{(AH)_2}
\begin{split}
 \langle A,H\rangle_{\rm HS}
 &=\frac{5}{8\pi}\sum_i
 \left(Cx_i,x_i(x^THx)-\frac25(Hx)_i\right)_D\\
 &\le \frac{5}{8\pi}\sum_i\|Cx_i\|_D
 \left\|x_i(x^THx)-\frac25(Hx)_i\right\|_D.
\end{split}
\end{equation}
Because $x^THx$ is a degree-two spherical harmonic,
$x_i(x^THx)-\frac25(Hx)_i$ is its degree-three component. Therefore,
using \eqref{eq:second-moment} and \eqref{eq:sixth-moment}, for each fixed $i$,
\begin{equation*}
	\begin{split}
		\left\|x_i(x^THx)-\frac25(Hx)_i\right\|_D^2
		&=12\int_{\mathbb S^2}
		\left(x_i(x^THx)-\frac25(Hx)_i\right)^2\,dv_{g_{\rm rd}}\\
		&=12\int_{\mathbb S^2}x_i^2(x^THx)^2\,dv_{g_{\rm rd}}
		-\frac{48}{25}\int_{\mathbb S^2}(Hx)_i^2\,dv_{g_{\rm rd}}\\
		&=12\cdot\frac{4\pi}{105}(2\|H\|^2_{\text{HS}}+8(H^2)_{ii})-\frac{48}{25}\cdot\frac{4\pi}{3}(H^2)_{ii}\\
		&=\frac{32\pi}{35}\|H\|_{\rm HS}^2
		+\frac{192\pi}{175}(H^2)_{ii}.
	\end{split}
\end{equation*}
Together with \eqref{(AH)_2} and
$\|Cx_i\|_D=\sqrt{\frac{8\pi}{3}(C^*C)_{ii}}$, this gives the refined estimate
\begin{equation}\label{(AH)_1}
 \langle A,H\rangle_{\rm HS}
 \le 2\sum_i\sqrt{(C^*C)_{ii}
 \left(\frac{5}{21}\|H\|_{\rm HS}^2+\frac27(H^2)_{ii}\right)}.
\end{equation}
Finally, Cauchy--Schwarz yields
\begin{equation*}
\begin{split}
 \langle A,H\rangle_{\rm HS}
 & \le 2 
 \left(\sum_i (C^*C)_{ii}\right)^{1/2} \left( \sum_i  \left(\frac{5}{21}\|H\|_{\rm HS}^2+\frac27(H^2)_{ii}\right) \right)^{1/2}\\
% &\le\frac{5}{8\pi}\left(\sum_i\|Cx_i\|_D^2\right)^{1/2}
% \left(\sum_i\left\|x_i(x^THx)-\frac25(Hx)_i\right\|_D^2\right)^{1/2}\\
 &=2\|C\|_{\rm HS}\|H\|_{\rm HS}.
\end{split}
\end{equation*}
Taking $H=A$ proves the proposition. 
\end{proof}

For simplicity, in what follows we denote the normalized degree-three component of $x_i(x^THx)$ by
\begin{equation*}
	U_i(H):=\frac{5}{\sqrt{96\pi}}\left(x_i(x^THx)-\frac{2}{5}(Hx)_i\right).
\end{equation*}
Then as we have seen in the proof of Proposition \ref{cor:mx-A-C}, 
\begin{equation*}
	\|U_i(H)\|^2_D=\frac{5}{21}\|H\|^2_{\text{HS}}+\frac{2}{7}(H^2)_{ii},\quad\sum_{i=1}^3\|U_i(H)\|^2_D=\|H\|^2_{\text{HS}}.
\end{equation*}

\section{Proof of Theorem \ref{main_thm_sphere_first}}
{\color{blue}\label{sec:sharp-first-remainder}}

We now apply the tools of Section~\ref{sec:prelim-sphere} to prove Theorem \ref{main_thm_sphere_first}. We will first prove the inequality  \eqref{improvedHersc_first}, then show that the   exponent $2$ and the constant $\frac{9}{160}$ are sharp.

\subsection{Proof of the inequality \eqref{improvedHersc_first}}

In this section, we set $H=\text{diag}\left(\frac{2}{3},-\frac{1}{3},-\frac{1}{3}\right)$ and denote $U_i=U_i(H)$.  It is easy to compute that
\begin{equation*}
	\|U_1\|_D^2=\frac{2}{7},\quad \|U_2\|_D^2=\|U_3\|_D^2=\frac{4}{21},\quad
	(U_i,U_j)_D=0,i\neq j.
\end{equation*}
To construct trial triple, we set two parameters $t,s$ and define
\begin{equation*}
	\phi_1=\frac{e_1+tU_1}{\|e_1+tU_1\|_D},\quad
	\phi_i=\frac{e_i+sU_i}{\|e_i+sU_i\|_D},i=2,3,
\end{equation*}
where $e_i=\sqrt{\frac{3}{8\pi}}x_i$ for $i=1,2,3$. These three functions are Dirichlet orthonormal, so we have
\begin{equation*}
	\mathcal{H}(\mu)+\frac{3}{2}\geq\sum^3_{i=1}(K_{\mu}(\phi_i-\overline{\phi}_i), \phi_i-\overline{\phi}_i)_D=\mathcal{F}_{t,s}-\mathcal{E}_{t,s}
\end{equation*}
where
\begin{equation*}
	\begin{split}
		\mathcal{F}_{t,s}&=\frac{1}{4\pi}\int_{\mathbb{S}^2}\frac{25}{24}\left(\frac{x_1^2(\frac{6}{5}+t(x_1^2-\frac{3}{5}))^2}{1+\frac{2}{7}t^2}+\frac{(1-x_1^2)(\frac{6}{5}+s(x_1^2-\frac{1}{5}))^2}{1+\frac{4}{21}s^2}\right)d\mu\\
		\mathcal{E}_{t,s}&=\frac{1}{4\pi}\left(\frac{t^2}{1+\frac{2}{7}t^2}\left(\int_{\mathbb{S}^2}U_1d\mu\right)^2+\frac{s^2}{1+\frac{4}{21}s^2}\sum_{j=2,3}\left(\int_{\mathbb{S}^2}U_jd\mu\right)^2\right)
	\end{split}
\end{equation*}
In particular, we set 
\begin{equation*}
	F_{t,s}(z)=\frac{25}{24}\left(\frac{z(\frac{6}{5}+t(z-\frac{3}{5}))^2}{1+\frac{2}{7}t^2}+\frac{(1-z)(\frac{6}{5}+s(z-\frac{1}{5}))^2}{1+\frac{4}{21}s^2}\right).
\end{equation*}
\begin{lemma}
	Suppose $-\frac{1}{2}\leq s\leq\frac{4}{7}$. Then $F_{s,s}$ is convex on $[0,1]$, nondecreasing on $[\frac{1}{3},1]$ if $s>0$, and
	nonincreasing on $[0,\frac{1}{3}]$ if $s<0$.
\end{lemma}
\begin{proof}
	The coefficient of $z^3$ in $F_{s,s}$ is always non-positive, and if $-\frac{1}{2}\leq s\leq\frac{4}{7}$,
	\begin{equation*}
		F_{s,s}''(1)=-\frac{35s^2(4s^2+8s-7)}{4(2s^2+7)(4s^2+21)}\geq 0.
	\end{equation*}
	So $F''_{s,s}(z)\geq F''_{s,s}(1)\geq 0$ for $0\leq z\leq 1$, i.e. $F_{s,s}$ is convex on $[0,1]$. Also, we have
	\begin{equation*}
		F_{s,s}'\left(\frac{1}{3}\right)=\frac{7s(10s^3+156s^2-33s+567)}{18(2s^2+7)(4s^2+21)}.
	\end{equation*}
	The bracket in the numerator is always positive on $[-\frac{1}{2},\frac{4}{7}]$, so the convexity gives the stated monotonicity.
\end{proof}

Differentiation shows that $F''_{t,s}(1)=0$ is equivalent to
\[
R_t(s):=4(t^2-4t+14)s^2+12(2t^2+7)s-21t(3t+4)=0.
\]
In particular, for any $t>0$, there is a unique positive $s=s(t)$ such that $F''_{t,s}(1)=0$. 
	\begin{lemma}\label{estimates for s(t)}
     Suppose $0<t<\frac{18}{35}$. Then 
		\begin{enumerate}
			\item $t<s(t)<\frac{11}{10}t$;
			\item  $F_{t,s(t)}$ is convex on $[0,1]$ and increasing on
			$[\frac13,1]$.
			\item 	Moreover, if $0<t<\frac{1}{10}$, then
			$t+\frac{t^2}{12}-\frac{t^3}{4}<s(t)<\frac{11}{10}t$. 
		\end{enumerate}  
	\end{lemma}
	
	\begin{proof}
				For fixed $t$, $R_t$ is strictly increasing on $s\geq0$, because
		\[
		R_t'(s)=8(t^2-4t+14)s+12(2t^2+7)>0.
		\]
		For $t\leq18/35$,
		\[
		R_t(t)=t^2(4t^2+8t-7)<0,
		\quad
		R_t\left(\frac{11}{10}t\right)
		=\frac{t}{25}(121t^3+176t^2+119t+210)>0,
		\]
		which proves (1).
		
		For (3), set $\delta=\frac{t^2}{12}-\frac{t^3}{4}$. Then
		\begin{equation*}
			\begin{split}
				R_t(t+\delta)
				&=R_t(t)+4(2t^3-2t^2+28t+21)\delta
				+4(t^2-4t+14)\delta^2\\
				&=-\frac{t^3}{3}(6t^3-8t^2+74t+11)
				+4(t^2-4t+14)\delta^2\\
				&\leq-\frac{10t^3}{3}+\frac{7t^4}{18}<0
				\quad\left(0<t\leq\frac{1}{10}\right),
			\end{split}
		\end{equation*}
		where we used $\delta\leq\frac{t^2}{12}$. Since $R_t$ is increasing, (3) follows.
		
		Finally we prove (2). Since $s(t)>t$,
		\begin{equation*}
			F''_{t,s(t)}(z)
			=\frac{25}{4}(z-1)\left(\frac{t^2}{1+\frac{2t^2}{7}}
			-\frac{s(t)^2}{1+\frac{4s(t)^2}{21}}\right)\geq0,
			\quad 0\leq z\leq1.
		\end{equation*}
		Thus $F_{t,s(t)}$ is convex. For the monotonicity, we first calculate 
		\[F'_{t,s(t)}\!\left(\frac13\right)
		=\frac{(9-2t)(3+t)}{18(1+2t^2/7)}
		-\frac{(9+s(t))(1-s(t))}{6(1+4s(t)^2/21)}.\]
		Write $G(s)=\frac{(9+s)(1-s)}{6(1+4\frac{s^2}{21})}$. By (1), 
		$0\leq s\leq\frac{11}{10}\cdot\frac{18}{35}$. So 
		\[
		G'(s)=\frac{7(16s^2-57s-84)}{(4s^2+21)^2}<0.
		\]
		So we obtain, using $s(t)>t$,
		\begin{equation*}
			\begin{split}
				F'_{t,s(t)}\!\left(\frac13\right)
				\ge F'_{t,t}\!\left(\frac13\right)
				=\frac{7t(10t^3+156t^2-33t+567)}
				{18(2t^2+7)(4t^2+21)}>0.
			\end{split}
		\end{equation*}
		Convexity now implies that $F_{t,s(t)}$ is increasing on $[1/3,1]$.
	\end{proof} 

\begin{lemma}
	Let $\mu$ be an admissible measure on $(\mathbb S^2,[g_{\rm rd}])$ with
	$\mu(\mathbb S^2)=4\pi$. Then
	\begin{equation*}
		\mathcal{E}_{t,s}\leq\frac{20}{27}\frac{\mathcal{H(\mu)}}{\lambda_1(\mu)}\max\left\{\frac{t^2}{1+\frac{2}{7}t^2},\frac{s^2}{1+\frac{4}{21}s^2}\right\}.
	\end{equation*}
\end{lemma}
\begin{proof}
	Using the formula of $U_i$ and the fact that $x_1^2-\frac{1}{3}$ is Dirichlet orthogonal to $E$, we have
	\begin{equation*}
		\frac{1}{4\pi}\left(\int_{\mathbb{S}^2}U_id\mu\right)^2
		\!=\!\frac{1}{4\pi}\frac{25}{36}\left(\!Ce_i,x_1^2\!-\!\frac{1}{3}\!\right)_D^2
		\!\leq\! \frac{1}{4\pi}\frac{25}{36}\|Ce_i\|^2_D\left\|x_1^2\!-\!\frac{1}{3}\right\|_D^2
		\!\leq\! \frac{10}{27}\|Ce_i\|^2_D.
	\end{equation*}
	In view of \eqref{eq:qs-A-small}, we have 
	\[\|Ce_i\|^2_D\leq  \|C\|^2_{\text{HS}}\leq \frac{2\mathcal{H}(\mu)}{\lambda_1(\mu)}.\]
	The conclusion follows from the definition of $\mathcal{E}_{t,s}$.
\end{proof}

\begin{proof}[Proof of the inequality \eqref{improvedHersc_first}]
	For brevity we omit $g$ and write $\lambda_i$ for $\lambda_i(g)$. 
	Using $\lambda_2\leq 4,\lambda_3\leq 6$, we have $\mathcal{H}(g)\geq\frac{1}{\lambda_1}-\frac{13}{12}.$
	Since the function $f(t)=\frac{1}{(t-2)^2}(\frac{1}{t}-\frac{13}{12})$ is decreasing on $(0,\frac{6}{7})$, we know 
	\begin{equation*}
		\frac{\mathcal{H}(g)}{(\lambda_1-2)^2}\geq f\left(\frac{6}{7}\right)=\frac{49}{768}>\frac{9}{160}
	\end{equation*}
	provided $\lambda_1\leq\frac{6}{7}$.
	
	It remains to consider $\frac{6}{7}\leq\lambda_1<2$. Set $0<t=\frac{9}{20}(2-\lambda_1)\leq\frac{18}{35}$ and assume, for contradiction, that $\mathcal{H}(g) < \frac{9}{160}(\lambda_1-2)^2$. Then
	\[
	B_{11}=\frac 32-B_{22}-B_{33} \ge \frac 32 -\frac 1{\lambda_2}-\frac 1{\lambda_3} = -\mathcal H(g) + \frac 1{\lambda_1} > \frac  1{\lambda_1}- \frac{9}{160}(\lambda_1-2)^2,
	\]
	and thus 	
	\begin{equation*}
		\begin{split}
			\frac{1}{4\pi}\int_{\mathbb{S}^2}\! x_1^2d\mu=\frac{2}{3}B_{11} > \frac{2}{3}\left(\frac{1}{\lambda_1}\!-\!\frac{9}{160}(\lambda_1\!-\!2)^2\right) =\frac{3}{9\!-\!10t}-\frac{5}{27}t^2>\frac{1}{3}+\frac{10}{27}t+\frac{2}{9}t^2.
		\end{split}
	\end{equation*}
	Write $m(t)=\frac{1}{3}+\frac{10}{27}t+\frac{2}{9}t^2$. The
		convexity and monotonicity of $F_{t,s(t)}$ give
	\begin{equation*}
		\mathcal{F}_{t,s(t)}\geq F_{t,s(t)}\left(\frac 1{4\pi}\int_{\mathbb{S}^2}x_1^2d\mu\right)\geq F_{t,s(t)}(m(t)).
	\end{equation*}
	Now the inequality \eqref{improvedHersc_first} follows from the following lemma (using the contradict assumption).  
	\end{proof}
	
	\begin{lemma} Suppose $0<t\leq\frac{18}{35}$, then
		\begin{equation}\label{inq_final}
			F_{t,s(t)}(m(t))-\mathcal{E}_{t,s(t)}-\frac{5}{18}t^2-\frac{3}{2}>0.
		\end{equation}
	\end{lemma}
	\begin{proof}
		Convexity and $\frac{1}{1+s}\geq 1-s$ gives
		\begin{equation*}
			\begin{split}
				F_{t,s(t)}(m(t))
				\geq\ &F_{t,s(t)}\left(\frac{1}{3}\right)+\left(\frac{10}{27}t+\frac{2}{9}t^2\right)F'_{t,s(t)}\left(\frac{1}{3}\right)\\
				\geq\ &\left(\frac{(9-2t)^2}{162}+\left(\frac{10}{27}t+\frac{2}{9}t^2\right)\frac{(9-2t)(3+t)}{18}\right)\left(1-\frac{2}{7}t^2\right)\\
				&+\frac{(9+s(t))((9-5t-3t^2)+(1+5t+3t^2)s(t))}{81}\left(1-\frac{4}{21}s(t)^2\right)\\
				=:\ &\Psi_t(s(t))+\frac{3}{2}+\frac{5}{18}t^2.
			\end{split}
		\end{equation*}
		
				We first show that $s\mapsto\Psi_t(s)$ is increasing on $t\leq s\leq\frac{11t}{10}$. 
		In fact
		\begin{equation}\label{Psits}
			\begin{split}
				\Psi_t'(s)=\ &\frac{1}{81}\left(18+40t+24t^2+\frac{2}{7}(-101+95t+57t^2)s\right. \\
				&\left.-\frac{4}{7}(18+40t+24t^2)s^2-\frac{16}{21}(1+5t+3t^2)s^3\right). 
			\end{split}
		\end{equation}
		%	From this expresssion we get $\Psi_t'(s)>0$, 
		Since 
		\[
		40t+\frac27(-101+95t+57t^2)s
		\ge 40t-\frac{202}{7}s \ge  \frac{289}{35}t>8t
		\]
		and
		\[
		\frac47(18+40t+24t^2)s^2
		+\frac{16}{21}(1+5t+3t^2)s^3 
		<
		\frac47\cdot 46 s^2
		+\frac{16}{21}\cdot \frac 92 s^3 < 32 t^2+5t^3, 
		\]
		we get 
		\[
		\Psi_t'(s) \ge \frac 1{81} (18+8t+24t^2-32 t^2 -5t^3) \ge \frac 29.
		\]

		If $0<t\leq\frac{1}{10}$,
		, then Lemma~\ref{estimates for s(t)} gives
		\begin{equation*}
			\Psi_t(s(t))\geq\Psi_t(t)+\frac{2}{9}\left(\frac{t^2}{12}-\frac{t^3}{4}\right)=\frac{t^3(907-478t-228t^2)}{3402}>\frac{t^3}{4}.
		\end{equation*}
		Meanwhile, using $s\leq\frac{11t}{10}$ again, we have
		\begin{equation*}
			\mathcal{E}_{t,s}\leq\frac{25t^2s^2}{27(9-10t)}<\frac{t^4}{7}\leq\frac{t^3}{70}.
		\end{equation*}
		This proves \eqref{inq_final} when $0<t\leq\frac{1}{10}$.
		
		Now suppose $\frac{1}{10}<t\leq\frac{18}{35}$.  
		Using the monotonicity of $\Psi_t$ and $s(t)>t$, we have
		\begin{equation*}
			\Psi_t(s(t))>\Psi_t(t)>t^2\left(\frac{8}{25}t-\frac{1}{54}-\frac{1}{5}t^2\right).
		\end{equation*}
	    Also $9-10t\geq\frac{27}{7}$ and $s(t)\leq\frac{11t}{10}$, so 
		\begin{equation*}
			\mathcal{E}_{t,s}<\frac{3t^4}{10}.
		\end{equation*}
		Consequently 
		\begin{equation*}
			\Psi_t(s(t))-\mathcal{E}_{t,s(t)}\geq t^2\left(\frac{8}{25}t-\frac{1}{54}-\frac{1}{2}t^2\right)>0
		\end{equation*}
		whenever $\frac{1}{10}<t\leq\frac{18}{35}$.  This completes the proof of \eqref{inq_final}. 
	\end{proof}

\subsection{Sharpness of the exponent and the coefficient}
\label{subsec:qs-sharpness}

We now show that the exponent two cannot be lowered, and that the coefficient $\frac {9}{160}$ is sharp. For $0\leq\varepsilon\ll1$, consider
\begin{equation}\label{eq:qs-perturbation}
	g_\varepsilon=(1+\varepsilon f)g_{\text{rd}},\qquad f=x_1^2-\frac{1}{3}.
\end{equation}
It is a smooth positive metric of area $4\pi$, and its area measure is Hersch-balanced since  
\begin{equation*}
	\int_{\mathbb{S}^2} x_i\,dv_{g_\varepsilon}=\varepsilon\int_{\mathbb{S}^2} x_ifdv_{g_{\text{rd}}}=0.
\end{equation*}
Let
\begin{equation}\label{eq:qs-psi}
	\psi_i=\sqrt{\frac{3}{4\pi}}x_i,\qquad i=1,2,3.
\end{equation}
Then $\psi_i$ is an $L^2(dv_{g_{\text{rd}}})$-orthonormal basis of the first eigenspace of $\Delta_{g_{\text{rd}}}$. 

The behavior of the eigenvalue problem of $g_{\varepsilon}$ is closely related to the  matrix 
\begin{equation}\label{eq:qs-splitting-matrix}
	M_{ij}=-2\int_{\mathbb{S}^2}f\psi_i\psi_jdv_{g_{\text{rd}}}.
\end{equation}
Standard perturbation theory for the multiple round eigenvalue shows that the one-sided derivatives at $\varepsilon=0^+$ of the three ordered eigenvalue branches are the ordered eigenvalues of $M$; see \cite{sou-ili-2,sou-ili-08}. A direct computation gives
\begin{equation*}
	M=\text{diag}\left(-\frac{8}{15},\frac{4}{15},\frac{4}{15}\right).
\end{equation*}
Furthermore, the first variation of $\mathcal H(g_\varepsilon)$ at $\varepsilon=0^+$ vanishes. The second-variation formula of \cite{Kar-25} gives
	\begin{equation*}
		\frac12\left.\frac{d^2}{d\varepsilon^2}\right|_{\varepsilon=0}\mathcal H(g_\varepsilon)
		=\sum_{i=1}^3 \sum_{\ell\ge2}
		\frac{\|\Pi_{\mathcal H_\ell}(f\psi_i)\|_{L^2}^2}{\ell(\ell+1)-2},
	\end{equation*}
	where $\Pi_{\mathcal H_\ell}$ denotes the $L^2$-orthogonal projection onto the spherical harmonics of degree $\ell$. Since $f$ has degree two, each product $f\psi_i$ has only degree-one and degree-three components. Thus only the degree-three component contributes, and using Lemma \ref{lem:momentsonsphere} one can calculate 
	\begin{equation*}
		\sum_{i=1}^3 \frac{\|\Pi_{\mathcal H_3}(f\psi_i)\|_{L^2}^2}{12-2}=\frac{2}{125}.
\end{equation*}

As a result, for $0\leq\varepsilon\ll 1$, we have
\begin{equation*}
	\lambda_1(g_{\varepsilon})=2-\frac{8}{15}\varepsilon+O(\varepsilon^2),\quad\lambda_i(g_{\varepsilon})=2+\frac{4}{15}\varepsilon+O(\varepsilon^2),i=2,3
\end{equation*}
and 
\begin{equation*}
	\mathcal{H}(g_{\varepsilon})=\frac{2}{125}\varepsilon^2+O(\varepsilon^3).
\end{equation*}
This proves the sharpness of the exponent $2$ and constant $\frac{9}{160}$.

\section{Quadratic reciprocal-sum stability and geometric stability}
\label{sec:quadratic-expanded}

 By uniformization
and diffeomorphism invariance, a smooth metric on $\mathbb{S}^2$ may be represented by
an admissible measure in the conformal class $[g_{\text{rd}}]$.  Hersch normalization
then gives \eqref{eq:qs-balance} without changing any eigenvalue. 

\subsection{Proof of Theorem \ref{main_thm_sphere}}
Let $A$ be the matrix from Section~\ref{subsec:matrix-prelim}. We first prove a quantitative improvement of \eqref{Hersch_mu}.
\begin{proposition}\label{Prop:Hlowerbdbytrace}
	Let $\mu$ be a balanced admissible measure on $\mathbb{S}^2$ with mass $4\pi$. If $A \ne 0$, then 
	\begin{equation*}
		\mathcal{H}(\mu)\geq\frac{(\text{\rm tr}A^2)^2}{2\left(\text{\rm tr}A^2+\frac{4}{7}\text{\rm tr}A^3\right)}.
	\end{equation*}
\end{proposition}
\begin{proof} 
 After a rotation
we may assume that $A$ is diagonal. 
%If $A=0$, the conclusion is immediate, so assume $A\ne0$. 
For $t\ge0$, set
\[
 v_i=e_i+tU_i(A),\qquad i=1,2,3.
\]
Since $A$ is diagonal, the functions $U_i(A)$ are mutually Dirichlet-orthogonal;
they are also orthogonal to $E$. Thus the three functions $v_i$ are mutually
Dirichlet-orthogonal. Moreover,
\[
 \|v_i\|_D^2
 =1+t^2\left(\frac{5}{21}\operatorname{tr}A^2
 +\frac27(A^2)_{ii}\right).
\]
By Cauchy--Schwarz for the positive form induced by $K_\mu$,
\begin{equation*}
\begin{split}
 (K_\mu v_i,v_i)_D
 &\ge (K_\mu e_i,e_i)_D+2t(K_\mu e_i,U_i(A))_D
 +t^2\frac{(K_\mu e_i,U_i(A))_D^2}{(K_\mu e_i,e_i)_D}\\
 &=\frac{(K_\mu e_i,e_i+tU_i(A))_D^2}{(K_\mu e_i,e_i)_D}.
\end{split}
\end{equation*}
The trace variational principle applied to the Dirichlet-orthogonal triple
$v_1,v_2,v_3$ therefore gives
\begin{equation*}
 \mathcal H(\mu)+\frac32
 \ge \frac{(\frac32+tS)^2}{\frac32+t^2D},
\end{equation*}
where
\[
 S:=\sum_{i=1}^3(K_\mu e_i,U_i(A))_D,
 \qquad
 D:=\sum_{i=1}^3(K_\mu e_i,e_i)_D\|U_i(A)\|_D^2.
\]
Here $D>0$ when $A\ne0$. Choosing $t=\frac{S}{D}$ gives the exact identity
\[
 \frac{(\frac32+tS)^2}{\frac32+t^2D}
 =\frac32+\frac{S^2}{D}.
\]
Furthermore,
\begin{equation*}
\begin{split}
 D
  =\sum_{i=1}^3\left(\frac12+A_{ii}\right)
 \left(\frac{5}{21}\operatorname{tr}A^2+\frac27(A^2)_{ii}\right) =\frac12\operatorname{tr}A^2+\frac27\operatorname{tr}A^3,
\end{split}
\end{equation*}
and the first identity in \eqref{(AH)_2} gives
\[
 S=\sum_{i=1}^3(Ce_i,U_i(A))_D=\frac12\operatorname{tr}A^2.
\]
Substitution yields
\[
 \mathcal H(\mu)\ge
 \frac{(\operatorname{tr}A^2)^2}
 {2\left(\operatorname{tr}A^2+\frac47\operatorname{tr}A^3\right)},
\]
as claimed. 
\end{proof}

\begin{proof}[Proof of Theorem \ref{main_thm_sphere}]
We need the following two lemmas. Let $a=(a_1,a_2,a_3)$ be a real triple with $a_i \ge -\frac 12$, and set
\begin{equation*} \Psi(a)=\frac{(a_1^2+a_2^2+a_3^2)^2}{2\left(a_1^2+a_2^2+a_3^2+\frac{4}{7}(a_1^3+a_2^3+a_3^3)\right)}\quad(a\ne0),\qquad \Psi(0):=0. 
\end{equation*}
 
\begin{lemma}\label{lem:Psi-Schur}
	The function $\Psi$ is Schur-convex on $
		\left\{a\in\mathbb{R}^3:\sum_i a_i=0,a_i\geq-\frac{1}{2}\right\}.$
\end{lemma}

\begin{proof}
Put $X=a_1^2+a_2^2+a_3^2,Y=a_1^3+a_2^3+a_3^3$. For distinct $i,j$ and the remaining index $k$,
\begin{equation*}
	(a_i-a_j)\left(\frac{\partial\Psi}{\partial a_i}-\frac{\partial\Psi}{\partial a_j}\right)
 	=\frac{X(a_i-a_j)^2}{(X+\frac{4}{7}Y)^2}\left(X+\frac87Y+\frac67Xa_k\right).
\end{equation*}
Using $a_i\geq-\frac{1}{2}$, the last bracket is nonnegative. The standard differential criterion \cite{MOA} proves Schur convexity away from the origin, and the definition $\Psi(0)=0$ extends it continuously to the whole stated domain.
\end{proof}

In the following lemma, we write $s=\frac{1}{2}-\frac{1}{\lambda_3},u=\frac{1}{2}-\frac{1}{\lambda_2}$.
\begin{lemma}\label{lem:nonlinear-water}
	Let $g$ be a Riemannian metric on $\mathbb S^2$ with
$\operatorname{Area}(\mathbb S^2,g)=4\pi$. Assume $\lambda_3(g)>2$. Then
\begin{enumerate}
	\item $\mathcal H(g)\ge \frac{21s^2}{4(7-2s)}$.
	\item Additionally if $\lambda_2(g)\ge2$, then $\mathcal H(g)\ge s^2+\frac57u(u+s)\geq\frac57(u+s)^2$.
	
\end{enumerate}
\end{lemma}

\begin{proof}
Set $a_i=A_{ii}$. Then $a_2\le-u$ and $a_3\le-s$. If $u\le-s/2$,
every admissible $a$ majorizes $(s/2,s/2,-s)$, whereas if $u\ge-s/2$ it
majorizes $(u+s,-u,-s)$. Proposition~\ref{Prop:Hlowerbdbytrace} and
Lemma~\ref{lem:Psi-Schur} therefore give
\begin{equation}\label{eq:nonlinear-water}
	\mathcal H(g)\geq
		\begin{cases}
			\displaystyle \phi(s):=\frac{21s^2}{4(7-2s)},&\displaystyle u\leq-\frac{s}{2},\\
			\displaystyle \Phi(s,u):=
			\frac{(u^2+us+s^2)^2}
			{u^2+us+s^2+\frac{6}{7}us(u+s)},&\displaystyle u\geq-\frac{s}{2}.
		\end{cases}
\end{equation}
To prove (1), one only needs to consider $-\frac{s}{2}\leq u<0$. Since $\lambda_3(g)\leq 6$ by \cite{knpp}, one has $0<s\leq\frac{1}{3}$. Direct subtraction gives
\begin{equation*}
	\Phi(s,u)-\phi(s)=\left(u+\frac{s}{2}\right)^2
	\frac{\frac{3s^2}{4}(1-\frac{10s}{7})
	+(1-\frac{2s}{7})(u+\frac{s}{2})^2}
	{(1-\frac{2s}{7})
	\left(\frac{3s^2}{4}(1-\frac{2s}{7})
	+(1+\frac{6s}{7})(u+\frac{s}{2})^2\right)}\geq0.
\end{equation*}

For (2), $u\geq0$, and hence $u(u+s)\geq0$. Using $s\leq\frac{1}{3}$ in the second
branch of \eqref{eq:nonlinear-water},
\begin{equation*}
	\Phi(s,u)=\frac{(s^2+u(u+s))^2}{s^2+u(u+s)+\frac{6}{7}su(u+s)}
			\geq\frac{(s^2+u(u+s))^2}{s^2+\frac{9}{7}u(u+s)}
			\geq s^2+\frac{5}{7}u(u+s),
\end{equation*}
where the last inequality is equivalent to
$\frac4{49}(u(u+s))^2\geq 0$. Finally, $s^2+\frac{5}{7}u(u+s)-\frac{5}{7}(u+s)^2=\frac{2}{7}s^2\geq0$. This completes the proof.
\end{proof}

Now we start to prove Theorem \ref{main_thm_sphere}.  For brevity, we omit $g$ in $\lambda_i(g)$ and $\mathcal H(g)$, and  set $\mathcal{S}:=\sum_{i=1}^3(\lambda_i-2)^2$. From \cite{knpp} we know $\lambda_i\leq 2i$, hence $\mathcal{S}\leq 24$. We will repeatedly use 
\begin{equation*}
	(\lambda_i-2)^2=4\lambda_i^2\left(\frac{1}{2}-\frac{1}{\lambda_i}\right)^2,i=1,2,3.
\end{equation*}
To achieve the conclusion, we split the proof into four cases.

\textbf{Case 1}: $\lambda_2 \ge 2$.  Assume the contrary $\mathcal{S}>183\mathcal{H}$.   Lemma \ref{lem:nonlinear-water} (2) gives
\begin{equation*}
	\mathcal{H}\geq\left(\frac{1}{2}-\frac{1}{\lambda_3}\right)^2+\frac{5}{7}\left(\frac{1}{2}-\frac{1}{\lambda_2}\right)\left(\left(\frac{1}{2}-\frac{1}{\lambda_2}\right)+\left(\frac{1}{2}-\frac{1}{\lambda_3}\right)\right)\geq\frac{5}{7}\left(\mathcal{H}+\frac{1}{2}-\frac{1}{\lambda_1}\right)^2.
\end{equation*}
Multiplying 144 on both sides of the first inequality gives
\begin{equation*}
	(\lambda_2-2)^2+(\lambda_3-2)^2\leq 64\left(\frac{1}{2}-\frac{1}{\lambda_2}\right)^2+144\left(\frac{1}{2}-\frac{1}{\lambda_3}\right)^2 \le  144\mathcal{H}.
\end{equation*}
Since $\mathcal{S}\leq 24$, we have $\mathcal{H}<\frac{2}{15}$ by the assumption. Thus 
\begin{equation*}
	(\lambda_1-2)^2\leq 16\left(\frac{1}{2}-\frac{1}{\lambda_1}\right)^2\leq 16\left(\mathcal{H}+\sqrt{\frac{7}{5}\mathcal{H}}\right)^2\leq 16\left(\sqrt{\frac{2}{15}}+\sqrt{\frac{7}{5}}\right)^2\mathcal{H}<39\mathcal{H}.
\end{equation*}
Hence $\mathcal{S}<(144+39)\mathcal{H}=183\mathcal{H}$, a contradiction.

\textbf{Case 2}: $\lambda_2<2$ and $\lambda_3\geq4$. In this case $(\lambda_i-2)^2\leq2\left(\frac{1}{\lambda_i} -\frac 12\right)$ for $i=1,2$. So
\begin{equation*}
	\mathcal{S}\leq 2 \left(\frac{1}{\lambda_1}-\frac{1}{2}\right)+ 2 \left(\frac{1}{\lambda_2}-\frac{1}{2}\right) +(\lambda_3-2)^2=2\left(\mathcal{H}+\left(\frac{1}{2}-\frac{1}{\lambda_3}\right)\right)+(\lambda_3-2)^2.
\end{equation*}
Consequently
\begin{equation*}
	\frac{\mathcal{S}}{\mathcal{H}}\leq 2+\frac{2\left(\frac{1}{2}-\frac{1}{\lambda_3}\right)+(\lambda_3-2)^2}{\phi\left(\frac{1}{2}-\frac{1}{\lambda_3}\right)}=:R\left(\frac{1}{2}-\frac{1}{\lambda_3}\right),
\end{equation*}
where
\begin{equation*}
	R(s)=2+\frac{8(7-2s)}{21s}+\frac{64(7-2s)}{21(1-2s)^2}.
\end{equation*}
A direct computation gives
\begin{equation*}
	R'(s)=\frac{8(2s+1)(12s^2+56s-7)}{21s^2(1-2s)^3}>0,
 	\qquad \frac{1}{4}\leq s\leq\frac{1}{3}.
\end{equation*}
Therefore $\frac{\mathcal{S}}{\mathcal{H}}\leq R\left(\frac{1}{3}\right)<183$.

\textbf{Case 3}: $\lambda_2<2$ and $2<\lambda_3\leq 4 $. Assume the contrary $\mathcal{S}>183\mathcal{H}$. By Lemma \ref{lem:nonlinear-water} (1),
\begin{equation*}
	\mathcal{H}\geq\phi\left(\frac{1}{2}-\frac{1}{\lambda_3}\right)\geq\frac{3}{4}\left(\frac{1}{2}-\frac{1}{\lambda_3}\right)^2
\end{equation*}
As above, using $\mathcal{H}<\frac{2}{15}$ again,
\begin{equation*}
	\begin{split}
		(\lambda_1-2)^2+(\lambda_2-2)^2&\leq 16\left(\frac{1}{2}-\frac{1}{\lambda_1}\right)^2+16\left(\frac{1}{2}-\frac{1}{\lambda_2}\right)^2\\
		&\leq16\left(\mathcal{H}+\left(\frac{1}{2}-\frac{1}{\lambda_3}\right)\right)^2\\
		&\leq 32\left(\mathcal{H}+\frac{4}{3}\right)\mathcal{H}<48\mathcal{H}.
	\end{split}
\end{equation*}
and
\begin{equation*}
	(\lambda_3-2)^2\leq 64\left(\frac{1}{2}-\frac{1}{\lambda_3}\right)^2<86\mathcal{H}.
\end{equation*}
Hence $\mathcal{S}<(48+86)\mathcal{H}=134\mathcal{H}$, a contradiction.

\textbf{Case 4}: $\lambda_3\leq 2$. We have
\begin{equation*}
 	\mathcal S\leq16\sum_i\left(\frac{1}{2}-\frac{1}{\lambda_i}\right)^2\leq16\left(\sum_{i}\left(\frac{1}{2}-\frac{1}{\lambda_i}\right)\right)^2=16\mathcal H^2.
\end{equation*}
Thus $\mathcal S<183\mathcal H$ when $\mathcal H\le1$, while for $\mathcal H>1$ the trivial bound $\mathcal S\le12$ suffices.
\end{proof}

\subsection{Proof of Theorem \ref{thm:sharp-reciprocal-hierarchy}}

We first keep the balancing normalization \eqref{eq:qs-balance}. The decomposition \eqref{decomposition_nu} of $d\nu$ and the definition of $A$ give
\begin{align*}
	d\nu=d\mu-2x^T\left(\frac12I+A\right)x\,dv_{g_{\rm rd}} =d\mu-(1+2x^TAx){dv_{g_{\rm rd}}}.
\end{align*}
In other words,
\begin{equation}\label{eq:qs-mu-decomp}
		d\mu-{dv_{g_{\rm rd}}}=d\nu+2x^TAx{dv_{g_{\rm rd}}}.
\end{equation}
We estimate the two terms $d\nu$ and $x^TAx{dv_{g_{\rm rd}}}$ separately.

First,  for every smooth $f$, we have
\begin{equation*}
	\left|\int_{\mathbb{S}^2} fd\nu\right|=\left|\sum_i(Cx_i,x_if)_D\right|
	\leq\sqrt{\frac{8\pi}{3}}\|C\|_{\mathrm{HS}}\left(\sum_i\|x_if\|_D^2\right)^{1/2}.
\end{equation*}
The multiplier term can be computed exactly.  Since
$\sum_i x_i^2=1$, $\sum_i x_i\,dx_i=\frac12d|x|^2=0$, and
$\sum_i|dx_i|^2=2$ pointwise, we have
\begin{equation*}
	\sum_i\|x_if\|_D^2=\int_{\mathbb{S}^2}\sum_i|x_i\,df+f\,dx_i|^2 dv_{g_{\rm rd}}=\int_{\mathbb{S}^2}(|df|^2+2f^2){dv_{g_{\rm rd}}}\leq2\|f\|_{W^{1,2}}^2.
\end{equation*}
Using \eqref{eq:qs-A-small}, it follows that
\begin{equation}\label{first_part}
\|d\nu\|_{W^{-1,2}}
	\le\sqrt{\frac{16\pi}{3}}\,\|C\|_{\mathrm{HS}}
	\le4\sqrt{\frac{2\pi}{3\lambda_1(\mu)}}\,\mathcal H(\mu)^{1/2}.
\end{equation}

Second, the fourth-moment formula \eqref{eq:qs-fourth} gives, for every
symmetric trace-free $A$, a function in $L^2$ defines a $W^{-1,2}$ distribution with norm at most its $L^2$ norm. Therefore, 
\begin{equation}\label{second_part}
	\begin{split}
	\|x^TAx\,{dv_{g_{\rm rd}}}\|_{W^{-1,2}} \leq\left(\int_{\mathbb{S}^2}(x^TAx)^2{dv_{g_{\rm rd}}}\right)^{1/2}
		=\sqrt{\frac{8\pi}{15}}\|A\|_{\mathrm{HS}} \leq8\sqrt{\frac{\pi}{15\lambda_1(\mu)}}\mathcal{H}(\mu)^{1/2},
	\end{split}
\end{equation}
where the equality uses \eqref{eq:qs-fourth}, and the last inequality uses \eqref{eq:qs-A-small} and Proposition \ref{cor:mx-A-C}.

Combining \eqref{eq:qs-mu-decomp}, \eqref{first_part}, and \eqref{second_part}, and using $\left(4\sqrt{\frac{2\pi}{3}}+16\sqrt{\frac{\pi}{15}}\right)^2<172$, we obtain
\begin{equation}\label{balance_position}
	\|\mu-dv_{g_{\text{rd}}}\|_{W^{-1,2}}\leq\sqrt{\frac{172}{\lambda_1(\mu)}}\mathcal{H}(\mu)^{1/2}.
\end{equation}
Finally, apply Hersch's lemma to the original measure and use the invariance
\eqref{eq:qs-invariance}. This gives a conformal automorphism $\Phi$ for which
$\Phi_*\mu$ is balanced. Since $\mathcal H(\mu)\leq\frac12$ implies
$\lambda_1(\mu)\geq\frac12$, \eqref{balance_position} yields
\eqref{eq:hierarchy-geo-intro}.

\section{The sharp Nadirashvili corner and the explicit double bubble}\label{proof of second eigenvalue}
\label{sec:corner-double-bubble}

The purpose of this section is to prove Theorem \ref{main_stability_second}.  First, we apply the
second-eigenvalue stability argument in  \cite[Proposition 5.1]{knps} to prove the quantitative 
Nadirashvili inequality \eqref{improvedNadir}.  Second, we introduce the explicit
smooth two-bubble metrics of Petrides \cite{pet25} to show the sharpness of the constant 2 in the exponential in \eqref{improvedNadir}.

\subsection{A refinement of Nadirashvili's bound}

We first establish the following estimate that gives an explicit relation between the first two eigenvalues. It will also be used for the numerical comparison below.

\begin{proposition}\label{prop:finite-scale-capacity}
Let $\mu$ be an admissible measure on $(\mathbb S^2,g_{\rm rd})$ with mass $4\pi$, and set $d=4-\lambda_2(\mu)$. If $0<d\le1/25$, then
\begin{equation}\label{eq:capacity-exp}
 \frac{1}{\lambda_1(\mu)}\ge J(d),
\end{equation}
where $J$ is defined in \eqref{lower_bounds_1/lambda_1}.
\end{proposition}

\begin{proof}
We follow the proof of \cite[Proposition 5.1]{knps}. After a conformal automorphism $\Phi$, there is a cap $Z=Z_r(p)$  and a conformal reflection $\tau=\tau_{\partial Z}:Z\to\mathbb{S}^2\setminus Z$ such that, for the folded
map
\begin{equation*}
	R_Z(x)=
		\begin{cases}
		x,& x\in\mathbb S^2\setminus Z,\\
		\tau(x),&x\in Z,
		\end{cases}
\end{equation*}
one has
\begin{align}
	 \operatorname{Area}(Z,g_{\rm rd}) &\leq\pi(4-\lambda_2(\mu))\label{eq:KNPS-explicit-residual}\\ 
	\|d\nu_Z-\lambda_2(\mu)\Phi_*\mu\|_{((C^0)\cap W^{1,2})^*}&\leq4\pi\sqrt{3(4-\lambda_2(\mu))},\label{eq:KNPS-explicit-residual-2}
\end{align}
where $d\nu_Z=|dR_Z|_{g_{\rm rd}}^2\,dv_{g_{\rm rd}}$ and $((C^0)\cap W^{1,2})^*$ stands for the dual norm of
\[\|u\|_{(C^0)\cap W^{1,2}}^2=\|u\|^2_{C^0}+\|du\|^2_{L^2}.\]
 
Under the hypothesis $d=4-\lambda_2(\mu)\le1/25$, \eqref{eq:KNPS-explicit-residual} implies $r<\pi/2$.  Let $\theta=d(p,\cdot)$ be the distance parameter based on $p$, and define radial function $h$ on $\mathbb{S}^2\setminus Z$ as
\begin{equation*}
	h(x)=
		\begin{cases}
			\displaystyle
			\frac{\log\tan\frac{r}{2}-\log\tan\frac{\theta}{2}}{\log\tan\frac{r}{2}},
			&r\leq\theta\leq\frac{\pi}{2},\\
			1,&\frac{\pi}{2}\leq\theta\leq\pi.
		\end{cases}
\end{equation*}
 We also put
\begin{equation*}%\label{eq:odd-condenser}
	f(x)=
		\begin{cases}
			h(x),&x\in\mathbb{S}^2\setminus Z,\\
			-h(\tau(x)),&x\in Z.
		\end{cases}
\end{equation*}
Then $f=0$ on $\partial Z$, and conformal invariance of the Dirichlet integral together with the change of variables under $\tau$ gives the exact identities
\begin{equation*}%\label{eq:integral f}
	\int_{\mathbb{S}^2}fd\nu_Z=0,\quad
	\int_{\mathbb{S}^2}|df|_{g_{\rm rd}}^2\,dv_{g_{\rm rd}}=-\frac{4\pi}{\log\tan\frac{r}{2}},\quad
	\int_{\mathbb{S}^2}f^2d\nu_Z=4\int_{\mathbb{S}^2\setminus Z}{h^2\,dv_{g_{\rm rd}}}.
\end{equation*}
As a consequence, 
\begin{equation}\label{Xnorm-f}
 	\|f\|_{(C^0)\cap W^{1,2}}\leq\left(1-\frac{4\pi}{\log\tan\frac{r}{2}}\right)^{1/2},
 	\qquad
 	\|f^2\|_{(C^0)\cap W^{1,2}}\leq\left(1-\frac{16\pi}{\log\tan\frac{r}{2}}\right)^{1/2}.
\end{equation}
Using \eqref{eq:KNPS-explicit-residual-2} and \eqref{Xnorm-f}, the variance of $f$ with respect to $\tilde\mu=\Phi_*\mu$ satisfies
\begin{equation*}%\label{eq:variance-explicit}
	\begin{split}
		\text{Var}_{\tilde\mu}(f):=&\ \int_{\mathbb{S}^2} f^2d\widetilde\mu
		-\frac1{4\pi}\left(\int f\,d\widetilde\mu\right)^2\\
		\geq&\ \frac{1}{\lambda_2(\mu)}\left(4\int_{\mathbb{S}^2\setminus Z}h^2\,dv_{g_{\rm rd}}-4\pi\sqrt{3(4-\lambda_2(\mu))}\left(1-\frac{16\pi}{\log\tan\frac{r}{2}}\right)^{1/2}\right)\\
		&\ -\frac{16\pi(3(4-\lambda_2(\mu)))}{4\lambda_2^2(\mu)}\left(1-\frac{4\pi}{\log\tan\frac{r}{2}}\right).\\
	\end{split}
\end{equation*} 

To estimate $\int_{\mathbb{S}^2\setminus Z}h^2\,dv_{g_{\rm rd}}$, we use the fact $\int_{0 \le \theta \le \frac{\pi}2} (-\log \tan \frac{\theta}2) dv_{\rm rd} = 2\pi \log 2$ and $h^2\geq2h-1$, 
\begin{equation*}%\label{eq:h2-lower}
	\int_{\mathbb{S}^2\setminus Z} h^2\,dv_{g_{\rm rd}}\geq
	\int_{\mathbb{S}^2\setminus Z} (2h-1)\,dv_{g_{\rm rd}}\geq4\pi- \operatorname{Area}(Z,g_{\rm rd}) +\frac{4\pi\log2}{\log\tan\frac{r}{2}}.
\end{equation*}
Since $-\log\tan\frac{r}{2}\geq\frac{1}{2}\log\frac{4-d}{d}$, we obtain
\begin{equation}\label{eq:RdDef}
	\begin{split}
		\frac{1}{4\pi}\text{Var}_{\tilde\mu}(f)&\geq
		1-\frac{d}{4}-\frac{2\log 2}{\log\frac{4-d}{d}}-\frac{\sqrt{3d}}{4\!-\!d}\left(1+\frac{32\pi}{\log\frac{4-d}{d}}\right)^{1/2}\!-\!\frac{3d}{(4\!-\!d)^2}\left(1\!+\!\frac{8\pi}{\log\frac{4-d}{d}}\right)\\
		&=:1-\frac{2\log 2}{\log\frac{4-d}{d}}-R(d).
	\end{split}
\end{equation}
\iffalse
It remains to make the universal coefficient explicit.  For $0<d\leq 1/25$ one has $\log\frac{4-d}{d}\geq\log99$ and consequently, 
\[
{\color{red}-\frac{d}{4}-\frac{\sqrt{3d}}{4-d}\left(1+\frac{32}{\log\frac{4-d}{d}}\right)^{1/2}-\frac{3d}{(4-d)^2}\left(1+\frac{8}{\log\frac{4-d}{d}}\right)
\ge -2.1\sqrt{d}-1.49d.}
\]
\fi
By variational principle, we have
\begin{equation}\label{lower_bounds_1/lambda_1}
	\begin{split}
		\frac{1}{\lambda_1(\mu)}\geq\frac{\text{Var}_{\tilde{\mu}}(f)}{\|df\|^2_{L^2}}&\geq\frac{1}{2}\log\frac{4-d}{d}\left(1-\frac{2\log 2}{\log\frac{4-d}{d}}-R(d)\right)\\
		&\geq\frac{1}{2}\log\frac{4-d}{d}-\log 2-\frac{1}{2}R(d)\log\frac{4-d}{d}  =:J(d).\\
	\end{split}
\end{equation}
This proves the proposition. 
\end{proof}

Given any $\lambda_1 \in (0, 2]$, one may solve \eqref{eq:capacity-exp} via {MATLAB} to get an upper bound of $\lambda_2$. Here are the numerical results which were used to plot the green curve near the horizontal dashed line $\lambda_2=4$ in Figure \ref{fig:joint-region-intro}.

\begin{table}[htbp]
	\begin{equation*}
    \begin{array}{cc|cc|cc}
    \lambda_1&\lambda_2\leq&\lambda_1&\lambda_2\leq&\lambda_1&\lambda_2\leq\\ \hline
		0.15&3.999998&0.80&3.982117&1.45&3.966751\\
		0.20&3.999959&0.85&3.980507&1.50&3.965955\\
		0.25&3.999734&0.90&3.978980&1.55&3.965197\\
		0.30&3.999141&0.95&3.977535&1.60&3.964475\\
		0.35&3.998111&1.00&3.976169&1.65&3.963787\\
		0.40&3.996702&1.05&3.974878&1.70&3.963130\\
		0.45&3.995023&1.10&3.973657&1.75&3.962503\\
		0.50&3.993180&1.15&3.972503&1.80&3.961903\\
		0.55&3.991262&1.20&3.971411&1.85&3.961330\\
		0.60&3.989331&1.25&3.970377&1.90&3.960780\\
		0.65&3.987430&1.30&3.969398&1.95&3.960253\\
		0.70&3.985584&1.35&3.968469&\\
		0.75&3.983811&1.40&3.967588&&\\\hline
    \end{array}
\end{equation*}
\caption{Numerical results from Proposition \ref{prop:finite-scale-capacity}.}
\label{Numerical result of improved nad}
\end{table}

In particular, when $\lambda_1=0.85$, Proposition~\ref{prop:finite-scale-capacity} gives numerically $\lambda_2\leq3.980507$. This is sharper than the estimate from \eqref{improvedNadir}, which gives $\lambda_2\leq3.989661$.

\begin{proof}[Proof of the inequality \eqref{improvedNadir}]
	For $0\leq d\leq\frac{1}{25}$, we have $\log\frac{4-d}{d}\geq\log 99$, then
	\begin{equation*}
		\left(1+\frac{32\pi}{\log\frac{4-d}{d}}\right)^{1/2}<4.79,\quad\left(1+\frac{8\pi}{\log\frac{4-d}{d}}\right)<6.47,\quad 4-d\geq 3.96.
	\end{equation*}
	Hence, the function $R(d)$ defined in \eqref{eq:RdDef} satisfies $R(d)\leq2.1 \sqrt{d}+1.49d$. Consequently
\iffalse	
{\color{blue}Set
\[
 Q(d)=(2.10\sqrt d+1.49d)\log\frac{4-d}{d}.
\]
For $0<d\le1/25$, using $\log((4-d)/d)\ge\log99$ and $4-d\ge3.96$ gives
\begin{equation*}
 Q'(d)\ge
 \frac{1}{\sqrt d}\left(1.05\log99-\frac{8.4}{3.96}\right)
 +\left(1.49\log99-\frac{5.96}{3.96}\right)>0.
\end{equation*}
Thus}
\fi
	\begin{equation*}
		R(d)\log\frac{4-d}{d}\leq\left(2.1\sqrt{\frac{1}{25}}+1.49\cdot\frac{1}{25}\right)\log\frac{4-\frac{1}{25}}{\frac{1}{25}}
        =\frac{1199}{2500}\log99<\frac{1199}{2500}\log100.
	\end{equation*}
	As a result
	\begin{equation*}
		d\exp\left\{\frac{2}{\lambda_1(\mu)}\right\}\geq\frac{4-d}{4}\left(\frac{4-d}{d}\right)^{-R(d)}\geq\frac{99}{100}100^{-\frac{1199}{2500}}>\frac{e}{25}.
	\end{equation*}
	If $d\geq\frac{1}{25}$, Hersch's inequality gives $\lambda_1\leq 2$, and therefore
	\begin{equation*}
		d\exp\left\{\frac{2}{\lambda_1(\mu)}\right\}\geq\frac{e}{25}.
	\end{equation*}
Equality in this last estimate would require simultaneously $d=1/25$ and $\lambda_1=2$, which is impossible. Hence the inequality in \eqref{improvedNadir} is strict. 
	This proves the desired inequality.
\end{proof}

\subsection{Petrides' explicit double-bubble metric}\label{subsec:petrides-double-bubble}

Petrides \cite{pet25} introduced the following family as a test family for the gap condition in the maximization of $\bar\lambda_1+t\bar\lambda_2$. His metric has area $8\pi$; we divide it by $2$ so that our normalization is $4\pi$. Set 
\begin{equation*}
	\beta_\varepsilon=\frac{1+\varepsilon^2}{1-\varepsilon^2} \;
	(0<\varepsilon<1), \quad \text{and}\quad \rho_\varepsilon^\pm(x_3)
	=\frac{\beta_\varepsilon^2-1}
	{(\beta_\varepsilon\pm x_3)^2}.
\end{equation*}
Let $\rho_\varepsilon
=\frac12\bigl(\rho_\varepsilon^+
+\rho_\varepsilon^-\bigr)$ and define 
\begin{equation*}
	g_\varepsilon=\rho_\varepsilon(x_3) g_{\rm rd}.
\end{equation*}
Each summand is the conformal factor of a round metric transported by a
M\"obius transformation.  Consequently
\begin{equation}\label{fixed_area}
 	\int_{\mathbb{S}^2}{ \rho_\varepsilon\,dv_{g_{\rm rd}}}=4\pi.
\end{equation}
As $\varepsilon\to 0$, the two terms concentrate at opposite poles, representing the two-bubble limit. When $\varepsilon\to 1$, $g_{\varepsilon}$ converges to the round metric.

According to \cite[{Theorem 0.4}]{pet25}, after the area
renormalization above,  
\begin{equation*}
	\lambda_1(g_\varepsilon)=\frac1{\log(1/\varepsilon)}+O\!\left(\frac1{\log^2(1/\varepsilon)}\right),
	\qquad
	\lambda_2(g_\varepsilon)=4-12\varepsilon^2+o(\varepsilon^2).
\end{equation*}
Therefore
\begin{equation}\label{eq:petrides-eliminate-eps}
	4-\lambda_2(g_\varepsilon)
	=\exp\!\left(-\frac{2+o(1)}{\lambda_1(g_\varepsilon)}\right).
\end{equation}
This completes the proof of Theorem \ref{main_stability_second}.

\iffalse
At the opposite endpoint one gets a useful perturbative check.  Let
$P_2(z)=(3z^2-1)/2$.  Expanding \eqref{eq:petrides-density-alpha} gives
\begin{equation}\label{eq:petrides-round-density-exp}
 \rho_\alpha=1+2\alpha^2P_2(z)+O(\alpha^4).
\end{equation}
First-order perturbation theory in the three-dimensional round first
eigenspace then yields
\begin{equation}\label{eq:petrides-round-spectrum-exp}
 \lambda_1(g_\alpha)
 =2-\frac85\alpha^2+O(\alpha^4),
 \qquad
 \lambda_2(g_\alpha)
 =2+\frac45\alpha^2+O(\alpha^4).
\end{equation}
Thus the double-bubble curve is tangent at $(2,2)$ to the line of slope
$-1/2$, exactly the tangent of Hersch's reciprocal-sum boundary.
\fi

\subsection{Monotonicity and numerical result}
Because $\rho_{\varepsilon}$ is rotationally symmetric, write
$u(x_3,\theta)=v(x_3)e^{im\theta}$. The eigenvalue equation for $g_\varepsilon$ is
\begin{equation}\label{reduced_ode}
 	-\frac d{dx_3}((1-x_3^2)v'(x_3))+\frac{m^2}{1-x_3^2}v(x_3)
 	=\lambda\rho_{\varepsilon}(x_3)v(x_3),\qquad -1<x_3<1,
\end{equation}
with the usual regularity condition at $x_3=\pm1$.  Hence an expansion in associated Legendre functions produces a symmetric generalized matrix eigenvalue problem.  The following values use the Legendre-Galerkin method, taking associated Legendre functions as base functions; they are included for orientation and for Figure~\ref{fig:joint-region-intro}.

\begin{table}[htbp]
	\centering 
	\begin{tabular}{c|cc@{\qquad}c|cc}
		$\varepsilon$&$\lambda_1$&$\lambda_2$&$\varepsilon$&$\lambda_1$&$\lambda_2$\\
		\hline
		$0.80$&$1.923696$&$2.039744$&$0.35$&$1.113667$&$2.801558$\\
		$0.70$&$1.817110$&$2.101145$&$0.30$&$0.992418$&$3.004651$\\
		$0.60$&$1.661246$&$2.205839$&$0.20$&$0.751997$&$3.475650$\\
		$0.50$&$1.462301$&$2.372995$&$0.15$&$0.633517$&$3.698263$\\
		$0.45$&$1.350685$&$2.488248$&$0.10$&$0.513863$&$3.868945$\\
		$0.40$&$1.233750$&$2.629891$&$0.05$&$0.384602$&$3.968856$\\ \hline 
	\end{tabular} 
	\vspace{.1cm} \caption{Eigenvalues in Petrides' double-bubble metric.}
	\label{Numerical result Petrides eigenvalues}
\end{table}

In fact, we can prove 

\begin{proposition}\label{monotonicity_of_petrides}
	For every $0<\varepsilon<1$, the first positive eigenvalue
	$\lambda_1(g_\varepsilon)$ is simple. % A corresponding eigenfunction 	can be chosen rotationally symmetric, odd with respect to the equator 	$\{x_3=0\}$, and strictly increasing as a function of 	$r=x_3\in(0,1)$. 
	Moreover, the function $\varepsilon\mapsto\lambda_1(g_\varepsilon)$ is strictly increasing on $(0,1)$.
\end{proposition}

\begin{proof}
  	We first show the first eigenvalue is solved by taking $m=0$ in \eqref{reduced_ode}. Equivalently, it is sufficient to prove that the first eigenfunctions are rotational symmetric.  Suppose the contrary that $m\geq 1$, then $u$ have zero angular
	average. Since $\rho_\varepsilon^\pm$ depend only on
	$r=x^3$, one has
	\[
	\int_{\mathbb{S}^2}u\rho_\varepsilon^{\pm}dv_{g_{\rm rd}}=0.
	\]
	Applying the variational principle to metrics $\rho_\varepsilon^{\pm}g_{\rm rd}$ and using the fact that $\lambda_1(\rho_\varepsilon^{\pm}g_{\rm rd})=\lambda_1(g_{\rm rd})=2$,
	\[
	\int_{\mathbb{S}^2}|du|_{g_{\rm rd}}^2dv_{g_{\rm rd}}
	\ge
	2\int_{\mathbb{S}^2}u^2\rho_\varepsilon^\pm dv_{\rm rd}.
	\]
	Adding the two inequalities yields
	\begin{equation*}
		\int_{\mathbb{S}^2}|du|_{g_{\rm rd}}^2dv_{g_{\rm rd}}
	\ge
	2\int_{\mathbb{S}^2}u^2\rho_\varepsilon dv_{g_{\rm rd}}.
	\end{equation*}
	This implies $\lambda_1(g_{\varepsilon})\geq 2$. On the other hand, if we put $t=\beta_\varepsilon^{-2}\in(0,1)$, then 
	\[
	\omega_t(r):=\rho_\varepsilon(r)
	=(1-t)\frac{1+tr^2}{(1-tr^2)^2}
	\]
    has the same monotonicity with the function $r^2$, since 
	\[
	\frac{\partial\omega_t}{\partial(r^2)}
	=(1-t)t\,\frac{3+tr^2}{(1-tr^2)^3}>0.
	\]
	Thus the strict Chebyshev inequality gives
	\[
	\int_{\mathbb{S}^2}r^2\omega_tdv_{\rm rd}
	> \frac 1{4\pi} \int_{\mathbb{S}^2}r^2 dv_{\rm rd}\int_{\mathbb{S}^2}\omega_tdv_{\rm rd} =
	\frac 1{4\pi} \cdot \frac{8\pi}{3}\cdot 4\pi
	=\frac{4\pi}{3}.
	\]
	Together with the fact that the radial function $r=x_3$ has
	$\rho_\varepsilon$-mean zero, we obtain
	\[
	\lambda_1(g_\varepsilon)
	\leq
	\frac{\displaystyle \int_{\mathbb{S}^2} |dr|^2dv_{\rm rd}}
	{\displaystyle \int_{\mathbb{S}^2} r^2\omega_tdv_{\rm rd}}
	<2.
	\]
	This contradiction shows that the first eigenfunction is rotationally symmetric.
	
	Consequently, the first eigenvalue is the first positive eigenvalue of the
	Sturm-Liouville problem
	\begin{equation}\label{reduced_ode_1}
		-((1-r^2)v'(r))'=\lambda_1(t)\omega_t(r)v(r),\qquad -1<r<1,
	\end{equation}
	and this eigenvalue is simple. Moreover, by $\mathbb{Z}_2$-symmetric trick, we know the eigenfunction $u$ of $\lambda_1$ is even or odd along $r$-direction. At the meantime, Courant's nodal domain theorem gives that $u$ has exactly two nodal domain on $\mathbb{S}^2$. So $v$ has only one zero at $r=0$, and it must be odd because the integral of $u$ over $dv_{g_{\varepsilon}}$ is $0$. Now we can suppose $v(r)>0$ for $0<r<1$. From \eqref{reduced_ode_1} we know $(1-r^2)v'(r)$ is strictly decreasing on $(0,1)$. It follows $(1-r^2)v'(r)>0$ and thus $v'(r)>0$ for $0<r<1$. 
	
	To prove the monotonicity of $\varepsilon\mapsto\lambda_1(g_{\varepsilon})$, we set $\lambda_1(t)=\lambda_1(g_{\varepsilon})$. It remains to prove $\lambda_1'(t)<0$. We  normalize $v=v_t$ such that $\|v\|_{L^2(\omega_tg_{\rm rd})}=1$.
	Use the first variation formula for Laplace eigenvalues, 
	\begin{equation}\label{first_variation}
		\lambda_1'(t)=-\lambda_1(t)
	\int_{\mathbb{S}^2}(\partial_t\omega_t)v_t^2dv_{g_{\rm rd}}.
	\end{equation}
	where by definition
	\[
	\partial_t\omega_t(r)
	=
	\frac{tr^4+3(1-t)r^2-1}{(1-tr^2)^3}.
	\tag{5.16}
	\]
	The numerator, regarded as a function of $r^2$, is strictly
	increasing, is negative at $r=0$, and equals $2(1-t)>0$ at $r=1$.
	Thus there is a unique $r_t\in(0,1)$ such that
	$\partial_t\omega_t(r_t)=0$. Since the total area is independent of $t$, we know $\int_{\mathbb{S}^2}\partial_t\omega_tdv_{g_{\rm rd}}=0$ and thus
	\[
	\int_{\mathbb{S}^2}\partial_t\omega_tv_t^2dv_{g_{\rm rd}}
	=
	\int_{\mathbb{S}^2}
	\partial_t\omega_t
	(v_t^2-v_t(r_t)^2)dv_{g_{\rm rd}}.
	\]
	By $v_t(r),v_t'(r)>0$ on $(0,1)$, we know $v_t^2$ is strictly increasing as a function of $|r|$. For $|r|<r_t$ both factors in the last integrand are negative, whereas for $|r|>r_t$ they are both positive. Hence
	\[
	\int_{\mathbb{S}^2}\partial_t\omega_t\,v_t^2\,dv_{g_{\rm rd}}>0.
	\]
	Equation \eqref{first_variation} therefore gives $\lambda_1'(t)<0$. This completes the proof.
\end{proof} 

\section{The diagonal family and double-bubble filling}
\label{sec:joint-range}

In this section we construct explicit families of admissible measures so that the joint first-two-eigenvalue fills  the region $\Omega$ in Figure \ref{fig:joint-region-intro}. 

\subsection{The diagonal lies in the spectral range}

We first construct admissible measures whose joint first-two-eigenvalue fills  the diagonal. 
For a center $p\in\mathbb S^2$ we use the logarithmic cutoff
\begin{equation}\label{cutoffu}
 u^{p}_{\epsilon,r}(x)=
 \begin{cases}
 1,&0\le d(x,p)\le\epsilon,\\
 \displaystyle\frac{\log(d(x,p)/r)}{\log(\epsilon/r)},&\epsilon<d(x,p)\le r,\\
 0,&d(x,p)>r,
 \end{cases}
\end{equation}
where $0<\epsilon<r$ and $d$ is the round distance.
\begin{proposition}\label{measures of multiplicity 2}
For every $\lambda\in(0,2]$ there exists a continuous conformal semimetric
\[
 h=\rho\,g_{\rm rd}, \qquad \rho\in C^0(\mathbb{S}^2),\quad \rho\ge0,
\]
whose zero set consists of at most two points, such that the area measure
${dv_h=\rho\,dv_{g_{\rm rd}}}$ is admissible, and 
\begin{equation*}
 dv_h(\mathbb{S}^2)=4\pi,
 \qquad
 \lambda_1(dv_h)=\lambda_2(dv_h)=\lambda.
\end{equation*}
\end{proposition}
\begin{proof}	
	We use  geodesic polar coordinates on $\mathbb{S}^2$ and let
	\begin{equation*}
		L_\alpha
		=\bigl\{(r,\theta)\in\mathbb{S}^2:
		0\leq r\leq\pi,\ 0\leq\theta\leq\alpha\bigr\}
	\end{equation*}
	denote the spherical lune. By \cite{Gromes}, the Dirichlet eigenvalues of $L_\alpha$ are
	\begin{equation*}
		\left(\frac{m\pi}{\alpha}+j\right)
		\left(\frac{m\pi}{\alpha}+j+1\right),
		\qquad m\in\mathbb{N}^*, \ j\in\mathbb{N};
	\end{equation*}
	while the Neumann eigenvalues are given by the same formula with $m\in\mathbb{N}$ instead.  The sphere $\mathbb{S}^2=L_{2\pi}$ is the union of three congruent copies of $L_{2\pi/3}$. We shall first construct a metric on $L_{2\pi/3}$ and then extend it to the entire sphere by rotations.
	
	Our construction uses the following two families of conformal maps.
	\begin{enumerate}
		\item For $0<\alpha\leq\beta<2\pi$, let
		\begin{equation*}
			\varphi_{\alpha,\beta}:L_\alpha\to L_\beta
		\end{equation*}
		be the conformal map between the corresponding spherical lunes. It is constructed by first applying stereographic projection, which maps $L_\alpha$ onto a planar sector of angle $\alpha$, then applying the power map $z\mapsto z^{\beta/\alpha}$, and finally applying the inverse stereographic projection.
		
		\item Let $H=Z_{\pi/2}$ be the hemisphere centered at 
		$p=\left(\frac{\pi}{2},\frac{\pi}{2}\right)$. 
		For $\pi/2\leq t<\pi$, let $\psi_t$ be the M\"obius transformation mapping $H$ conformally onto the spherical cap $Z_t$ centered at $p$.
	\end{enumerate}	 
	
	Combining these two families, for $2\pi/3\leq t<2\pi$ we define $\Phi_t:L_{2\pi/3}\to\mathbb{S}^2$ by
	\begin{equation*} 
		\Phi_t(r,\theta)=
		\begin{cases}
			\varphi_{2\pi/3,t}(r,\theta),
			&   \frac{2\pi}{3}\leq t\leq\pi,\\ 
			\bigl(\psi_{t/2}\circ\varphi_{2\pi/3,\pi}\bigr)(r,\theta),
			&   \pi\leq t<2\pi.
		\end{cases}
	\end{equation*}
	Each map $\Phi_t$ is conformal, and we set $g_t=\Phi_t^*g_{\rm rd}$
	on $L_{2\pi/3}$.
	The reflection
	\begin{equation*}
		R(r,\theta)=\left(r,\frac{2\pi}{3}-\theta\right)
	\end{equation*}
	is an isometry of $(L_{2\pi/3},g_t)$. Consequently, the three rotated copies of this metric, obtained by rotations through $2\pi/3$ and $4\pi/3$, agree along their common boundary meridians and define a metric on $\mathbb{S}^2$. We continue to denote the resulting metric by $g_t$.  
    We point out that the metric $g_t$ has two singularities, at $r=0$ and $r=\pi$. More precisely,
	\begin{equation*}
		{g_t=\rho_tg_{\rm rd},}
	\end{equation*}
	where $\rho_t$ is continuous on $\mathbb{S}^2$, smooth away from the two poles,
and $\rho_t(0,\theta)=\rho_t(\pi,\theta)=0$ unless $t=2\pi/3$, when
$g_t=g_{\rm rd}$.  (For $2\pi/3<t<\pi$, the power map generally gives only
H\"older regularity at the poles, so no Lipschitz regularity is needed here.)
Since $\rho_t$ is bounded, Rellich compactness for
$W^{1,2}(\mathbb{S}^2,g_{\rm rd})\hookrightarrow L^2(dv_{g_{\rm rd}})$ implies compactness
into $L^2(\rho_tdv_{g_{\rm rd}})$; hence $dv_{g_t}$ is admissible.

	Now we prove that up to area normalization, the metrics $g_t$ have all the required properties. Since $dv_{g_t}=\rho_t\,dv_{g_{\rm rd}}$ is an admissible measure in the conformal class $[g_{\rm rd}]$, we interpret
	\begin{equation*}
		\lambda_k(\mathbb{S}^2,g_t)
		:=\lambda_k(dv_{g_t}).
	\end{equation*}
	It remains to establish the following three assertions:
	\begin{enumerate}
		\item $4\pi\leq\text{Area}(\mathbb{S}^2,g_t)\leq12\pi$;
		\item $\lambda_1(\mathbb{S}^2,g_t)=\lambda_2(\mathbb{S}^2,g_t)$;
		\item $\lambda_1(\mathbb{S}^2,g_{2\pi/3})=2$ and $\lim_{t\to2\pi}\lambda_1(\mathbb{S}^2,g_t)=0$.
	\end{enumerate}

    The first assertion follows immediately from the construction.  To prove the second assertion, we use the $S_3$-symmetry of $(\mathbb{S}^2,g_t)$. The key observation is the following.
    
    \begin{lemma}\label{simple eigenvalue}
        If $\lambda_k(\mathbb{S}^2,g_t)$ is simple for some $t$, then $k\geq 3$.
    \end{lemma}
    \begin{proof}
    	Let $\phi$ be a real eigenfunction associated with the simple eigenvalue $\lambda_k(\mathbb{S}^2,g_t)$. The symmetry group $S_3$ is generated by the reflection $R$ and the rotation $T$ through the angle $2\pi/3$. Since the eigenspace spanned by $\phi$ is one-dimensional, it carries a one-dimensional real representation of $S_3$. By the classification of irreducible representations of $S_3$, there are therefore only two possibilities:
    	\begin{enumerate}
    		\item $R^*\phi=\phi$ and $T^*\phi=\phi$;
    		\item $R^*\phi=-\phi$ and $T^*\phi=\phi$.
    	\end{enumerate}
    	
    	In the first case, $\phi$ is invariant under all reflections in the boundary meridians of the three lunes. Its restriction to $L_{2\pi/3}$ is therefore a Neumann eigenfunction of $(L_{2\pi/3},g_t)$. If $2\pi/3\leq t\leq\pi$, then, for some $(m,j)\in\mathbb{N}^2\setminus\{(0,0)\}$,
    	\begin{align*}
    		\lambda_k(\mathbb{S}^2,g_t)
    		\text{Area}(\mathbb{S}^2,g_t) \!=\!6t\left(\frac{m\pi}{t}\!+\!j\right)
    		\left(\frac{m\pi}{t}\!+\!j\!+\!1\right) 
    		\!\geq\!
    		\min\left\{
    		6\pi\left(\frac{\pi}{t}\!+\!1\right), 12t
    		\right\} \!\geq \! 8\pi.
    	\end{align*}
    	If $\pi\leq t<2\pi$, the monotonicity result for the first nonzero Neumann eigenvalue of a spherical cap \cite[Theorem 2]{LangfordLaugesen} gives
    	\begin{align*}
    		\lambda_k(\mathbb{S}^2,g_t)
    		\text{Area}(\mathbb{S}^2,g_t) \geq
    		3\lambda_1^N(Z_{t/2},g_{\rm rd})
    	\operatorname{Area}(Z_{t/2},g_{\rm rd})\geq12\pi.
    	\end{align*}

    	In the second case, every reflection acts on $\phi$ by multiplication by $-1$. Thus the restriction of $\phi$ to $L_{2\pi/3}$ satisfies Dirichlet boundary conditions.  For $2\pi/3\leq t\leq\pi$, this means  that   there exists $(m,j)\in\mathbb{N}^*\times \mathbb{N}$  such that
    	\begin{align*}
    		\lambda_k(\mathbb{S}^2,g_t)
    		\text{Area}(\mathbb{S}^2,g_t) =6t\left(\frac{m\pi}{t}+j\right)
    		\left(\frac{m\pi}{t}+j+1\right) \geq 6\pi\left(\frac{\pi}{t}+1\right) \geq 12\pi.
    	\end{align*}
    	For $\pi\leq t<2\pi$, we observe that the fact   $R^*\phi=-\phi$ implies $\phi$ vanishes on the middle meridian of the lune. In particular,  $\phi$ cannot be the first Dirichlet eigenfunction of $(L_{\frac{2\pi}{3}},g_t)$. This fact, together with the domain monotonicity of the Dirichlet eigenvalues, gives 
    	\begin{align*}
    		\lambda_k(\mathbb{S}^2\!,g_t)
    		\text{Area}(\mathbb{S}^2,g_t) \!\geq\!
    		3\lambda_2^D(Z_{t/2},g_{\rm rd}) \!
    		\operatorname{Area}(Z_{t/2},g_{\rm rd}) \!\geq\!
    		6\pi 
    		\lambda_2^D(Z_{t/2},g_{\rm rd}) \!\geq\! 12\pi.
    	\end{align*}

    	We thus proved in both cases that
    	\begin{equation*}
    		\lambda_k(\mathbb{S}^2,g_t)
    		\text{Area}(\mathbb{S}^2,g_t)\geq8\pi.
    	\end{equation*}
    	Equality can occur only when $t=2\pi/3$, in which case $g_t=g$ is the round metric. Since every positive eigenvalue of the round sphere has multiplicity greater than one, we conclude 
    	\begin{equation*}
    		\lambda_k(\mathbb{S}^2,g_t)
    		\text{Area}(\mathbb{S}^2,g_t)>8\pi.
    	\end{equation*}
    Hersch's scale-invariant inequality
    \[
      \lambda_1(\mathbb{S}^2,g_t)\,\operatorname{Area}(\mathbb{S}^2,g_t)\le8\pi
    \]
    rules out $k=1$.  Hence the first positive eigenvalue is not simple and
    \(\lambda_1(\mathbb{S}^2,g_t)=\lambda_2(\mathbb{S}^2,g_t)\).  It follows in turn that a
    simple positive eigenvalue must have index $k\ge3$.
    \end{proof}

    The lemma proves the second assertion. We now prove the third. Since $g_{2\pi/3}=g_{\rm rd}$,  
    \begin{equation*}
    	\lambda_1(\mathbb{S}^2,g_{2\pi/3})=2.
    \end{equation*}
    It remains to establish the limit as $t\to2\pi$. Let $p$ be the center of the cap $Z_{t/2}$. Since
    $d(-p,\partial Z_{t/2})=\pi-t/2$, the cutoff $u^{-p}_{\pi-t/2,1}$ equals $1$ on $\partial Z_{t/2}$. Hence, for $2\pi-2\leq t<2\pi$, the function
    \begin{equation*}
    	u_t(x)=
    	\begin{cases}
    		1,&x\notin L_{2\pi/3},\\
    		\bigl(\Phi_t^*u^{-p}_{\pi-t/2,1}\bigr)(x),&x\in L_{2\pi/3},
    	\end{cases}
    \end{equation*}
    is continuous across the boundary meridians and belongs to $W^{1,2}(\mathbb S^2)$.  Let
    \begin{equation*}
    	c_t
    	=\frac{1}{\text{Area}(\mathbb{S}^2,g_t)}
    	\int_{\mathbb{S}^2}u_t\,dv_{g_t}.
    \end{equation*}
    Then  
    \begin{align*}
    	\int_{\mathbb{S}^2}|du_t|_{g_t}^2\,dv_{g_t}
    	=\int_{L_{2\pi/3}}|du_t|_{g_t}^2\,dv_{g_t} =\int_{Z_{t/2}}|du^{-p}_{\pi-t/2,1}|_{g_{\rm rd}}^2\,dv_{g_{\rm rd}} \leq\frac{-2\pi}{\log (\pi-t/2)}
    \end{align*}
    and  
    \begin{align*}
    	\int_{\mathbb{S}^2}(u_t\!-\!c_t)^2 dv_{g_t} \!
    	 \geq
    	2\text{Area}(L_{2\pi/3},g_t)(1\!-\!c_t)^2
    	\!+\!2\pi c_t^2  \!\geq\!
    	\frac{2}{
    		\frac{1}{\text{Area}(L_{2\pi/3},g_t)}
    		\!+\!\frac{1}{\pi}}
    	\!\geq\!\frac{8\pi}{7}.
    \end{align*}
    Therefore, by the variational characterization of the first nonzero eigenvalue,
    \begin{align*}
    	\lambda_1(\mathbb{S}^2,g_t)
    	 \leq
    	\frac{\int_{\mathbb{S}^2}|du_t|_{g_t}^2\,dv_{g_t}
    	}{\int_{\mathbb{S}^2}(u_t-c_t)^2\,dv_{g_t}
    	} \leq
    	\frac{7}{4\log\!\left(1/(\pi-t/2)\right)}
    	\to 0
    	\qquad\text{as }t\to2\pi.
    \end{align*}

    We finally normalize the area. Set $\widetilde g_t
    =\frac{4\pi}{\text{Area}(\mathbb{S}^2,g_t)}\,g_t$. Then
    \begin{equation*}
    	\text{Area}(\mathbb{S}^2,\widetilde g_t)=4\pi
    \end{equation*}
    and, for every $k$,
    \begin{equation*}
    	\lambda_k(\mathbb{S}^2,\widetilde g_t)
    	=\frac{\text{Area}(\mathbb{S}^2,g_t)}{4\pi}
    	\lambda_k(\mathbb{S}^2,g_t).
    \end{equation*}
    Thus, for every $2\pi/3\leq t<2\pi$,
    \begin{equation*}
    	\lambda_1(\mathbb{S}^2,\widetilde g_t)
    	=\lambda_2(\mathbb{S}^2,\widetilde g_t),
    \end{equation*}
    while
    \begin{equation*}
    	\lambda_1(\mathbb{S}^2,\widetilde g_{2\pi/3})
    	=\lambda_2(\mathbb{S}^2,\widetilde g_{2\pi/3})=2
    \end{equation*}
    and
    \begin{equation*}
    	\lim_{t\to2\pi}\lambda_1(\mathbb{S}^2,\widetilde g_t)
    	=\lim_{t\to2\pi}\lambda_2(\mathbb{S}^2,\widetilde g_t)=0.
    \end{equation*}
    The conclusion follows from the continuity of the map $
    t\longmapsto\lambda_k(\mathbb{S}^2,\widetilde g_t)$.  	We refer to \cite[Section~6]{knpp} for the continuity of eigenvalues associated with measures whose densities belong to $L^p$, $p>1$. 
\end{proof}

\subsection{Convex interpolation and a degree argument}

We now join the double-bubble family of Section \ref{subsec:petrides-double-bubble} to the diagonal
family of Proposition~\ref{measures of multiplicity 2}.  For
$u\in[0,1)$, extend the Petrides family continuously by setting $g_1:=g_{\rm rd}$ and let
\begin{equation*}
	\nu_u^-:=dv_{g_{1-u}}.
\end{equation*}
Thus $\nu_0^-=dv_{g_{\text{rd}}}$ and, as $u\to 1^-$, $(\lambda_1(\nu_u^-),\lambda_2(\nu_u^-))\to(0,4)$. Choose the continuous diagonal family produced in the preceding proof and 
parameterize it by the same interval $u\in[0,1)$, with
$\nu_0^+=dv_{g_{\text{rd}}}$ and $\lambda_1(\nu_u^+)=\lambda_2(\nu_u^+)\to 0,u\to 1^-$.

For $(\sigma,u)\in[0,1]\times[0,1)$ set
\begin{equation}\label{eq:convex-family}
 	\nu_{\sigma,u}:=(1-\sigma)\nu_u^-+\sigma\nu_u^+.
\end{equation}
Every measure in \eqref{eq:convex-family} has mass $4\pi$ and is admissible. The following lemma is needed.
\iffalse
Indeed, if a sequence is bounded in $W^{1,2}$, compactness of the two endpoint
embeddings gives, after taking a common subsequence, convergence in both
$L^2(\nu_u^-)$ and $L^2(\nu_u^+)$; the identity
\[
 \|f\|_{L^2(\nu_{\sigma,u})}^2
 =(1-\sigma)\|f\|_{L^2(\nu_u^-)}^2
 +\sigma\|f\|_{L^2(\nu_u^+)}^2
\]
then gives convergence in $L^2(\nu_{\sigma,u})$.
\fi

\begin{lemma}\label{lem:joint-uniform}
	We have the inequality
	\begin{equation}\label{eq:uniform-endpoint-bound}
		\sup_{0\le\sigma\le1}\lambda_1(\nu_{\sigma,u})\leq 2\max\{\lambda_1(\nu_u^-),\lambda_1(\nu_u^+)\}.
	\end{equation}
	In particular, the left-hand side tends to zero as $u\to 1^-$.
\end{lemma}

\begin{proof}
For a mass-$4\pi$ measure $\nu$ define
\begin{equation*}
	\operatorname{Var}_\nu(v):=\inf_{c\in\mathbb{R}}\int_{\mathbb{S}^2}(v-c)^2\,d\nu.
\end{equation*}
The first eigenvalue can be written as
\begin{equation}\label{eq:lambda-one-variance}
	\lambda_1(\nu)
	=\inf_{\operatorname{Var}_\nu(v)>0}
	\frac{\int_{\mathbb{S}^2}|dv|_{g_{\text{rd}}}^2\,dv_{g_{\text{rd}}}}
		{\operatorname{Var}_\nu(v)}.
\end{equation}
For every $v$,
\begin{equation*}
	\begin{split}
		\operatorname{Var}_{\nu_{\sigma,u}}(v)
 		&=\inf_c\left[(1-\sigma)\int(v-c)^2d\nu_u^-
        +\sigma\int(v-c)^2d\nu_u^+\right]\\
		&\geq\max\{(1-\sigma)\operatorname{Var}_{\nu_u^-}(v),
            \sigma\operatorname{Var}_{\nu_u^+}(v)\}.
	\end{split}
\end{equation*}
Let $\phi_u^-$ be a first eigenfunction normalized by
$\operatorname{Var}_{\nu_u^-}(\phi_u^-)=1$.  If $0\le\sigma\le1/2$, then
\[
 \operatorname{Var}_{\nu_{\sigma,u}}(\phi_u^-)
 \ge1-\sigma\ge\frac12,
\]
so \eqref{eq:lambda-one-variance} gives
$\lambda_1(\nu_{\sigma,u})\le2\lambda_1(\nu_u^-)$.  For
$1/2\le\sigma\le1$ the same argument with a first eigenfunction of $\nu_u^+$
gives $\lambda_1(\nu_{\sigma,u})\le2\lambda_1(\nu_u^+)$.  This proves
\eqref{eq:uniform-endpoint-bound}.
\end{proof}

Consider the continuous spectral map
\begin{equation*}
 \Lambda(\sigma,u)
 :=\bigl(\lambda_1(\nu_{\sigma,u}),\lambda_2(\nu_{\sigma,u})\bigr).
\end{equation*}
Continuity follows from the standard continuity theorem for eigenvalues of
admissible measures along these density families; see \cite[Section~6]{knpp}.
At $u=0$ both endpoint measures are round, so
$\Lambda(\sigma,0)=(2,2)$ for $0\le\sigma\le1$. The side $\sigma=0$ is the
Petrides family, whereas the side $\sigma=1$ is the diagonal family.

Let $\Omega$ be the domain bounded by
\begin{enumerate}
 \item the Petrides path $u\mapsto\Lambda(0,u)$ from $(2,2)$ to its limit $(0,4)$;
 \item the vertical segment from $(0,4)$ to $(0,0)$;
 \item the diagonal path $u\mapsto\Lambda(1,u)$ traversed in reverse, from $(0,0)$ to $(2,2)$.
\end{enumerate}
Indeed, by Proposition \ref{monotonicity_of_petrides} we know (1)-(3) forms a Jordan curve. Figure~\ref{fig:joint-region-intro} is a numerical
rendering of that situation.

\begin{proposition}\label{imageunderred}
Every point of $\Omega$ belongs to the joint spectral range $\mathscr R$.
Boundary points lying on the Petrides or diagonal arcs are realized by the
corresponding endpoint families.
\end{proposition}

\begin{proof}
By Lemma~\ref{lem:joint-uniform}, there exists $\delta>0$ sufficiently small, such that
\begin{equation}\label{eq:top-left-strip}
 \lambda_1(\nu_{\sigma,u})<\varepsilon
 \qquad\text{for every }(\sigma,u)\in[0,1]\times[1-\delta,1).
\end{equation}
Consider $\Lambda$ on the rectangle
$\Gamma_\delta=[0,1]\times[0,1-\delta]$. Its bottom side is collapsed to $(2,2)$; its left and right sides are the truncated Petrides and diagonal paths; and by \eqref{eq:top-left-strip} its top side lies in the half-plane
$\{x<\varepsilon\}$. 

Let $j_{\varepsilon}:[0,4]^2\to\Omega_{\varepsilon}$ be a strong deformation retraction, which fixed $\Omega_{\epsilon}$ and maps $[0,4]^2\setminus\Omega_{\epsilon}$ onto $\partial\Omega_{\varepsilon}$. Set $F_{\varepsilon}=j_{\varepsilon}\circ\Lambda$, which maps $\Gamma_{\delta}$ to $\overline{\Omega}_{\varepsilon}$ with non-zero degree on the boundary. Hence, $\Omega_{\varepsilon}$ belongs to the image of $F_{\varepsilon}$, and thus $\Omega_{\varepsilon}\subset\mathscr{R}$. Finally, let $\varepsilon\to 0$ we can conclude the result.
\end{proof} 

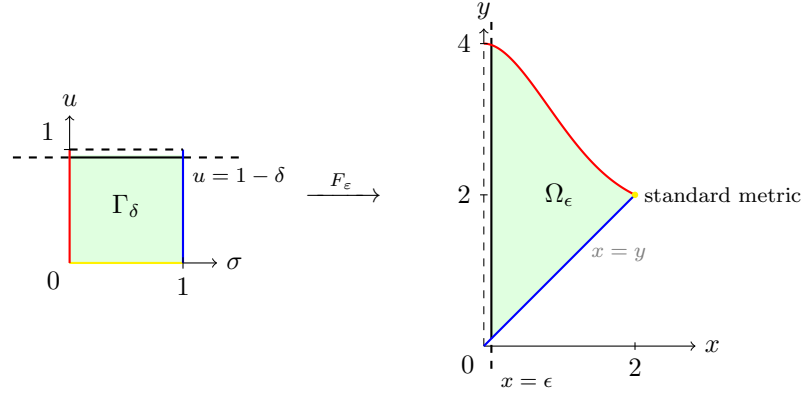
\begin{figure}[htbp] 
	\begin{minipage}{0.3\linewidth}
		\centering
		\begin{tikzpicture}[scale=1.5]
			% axis
			\draw[->] (0,0) -- (1.3,0) node[right] {$\sigma$};
			\draw[->] (0,0) -- (0,1.3) node[above] {$u$};
			
			\fill[green!12]
			plot (0,0) -- (1,0) -- (1,0.95) -- (0,0.95)
			-- cycle;
			\node at (0.5,0.5) {$\Gamma_{\delta}$};
			
			\draw[thick] (0,0.93) -- (1,0.93);
			\draw[thick, dashed] (-0.5,0.93) -- (0,0.93);
			\draw[thick, dashed] (1,0.93) -- (1.5,0.93);
			\node[below right] at (1,0.93) {\footnotesize $u=1-\delta$};
			
			\draw[thick,yellow] (0,0) -- (1,0);
			\draw[thick,red] (0,0) -- (0,1);
			\draw[thick,blue] (1,0) -- (1,1);
			\draw[dashed,thick] (0,1) -- (1,1);
			
			\foreach \x in {1} {\draw (\x,0.05) -- (\x,-0.05) node[below] {\x};}
			\foreach \y in {1} {\draw (0.05,\y) -- (-0.05,\y) node[above left] {\y};}
			\node[below left] at (0,0) {0};
		\end{tikzpicture}
	\end{minipage}
	$\xrightarrow{\ \ F_{\varepsilon}\ \ }$
	\begin{minipage}{0.5\linewidth}
		\centering
		\begin{tikzpicture}[scale=1]
			% axis
			\draw[->] (0,0) -- (2.8,0) node[right] {$x$};
			\draw[dashed] (0,0) -- (0,4);
			\draw[->] (0,4) -- (0,4.2) node[above] {$y$};
			\node[below left] at (0,0) {0};
			
			% the region
			\fill[green!12]
			plot[domain=0.1:2] (\x, \x) --                          
			% \lambda_1=\lambda_2       
			plot[domain=2:0.1](\x, {-0.16155*\x^4+1.021216*\x^3-1.896216*\x^2+4}) -- cycle;
			\node at (1,2) {$\Omega_{\epsilon}$};
			
			\draw[thick] (0.1,0.1) -- (0.1,3.97);
			\draw[thick, dashed] (0.1,-0.3) -- (0.1,0.1);
			\draw[thick, dashed] (0.1,3.97) -- (0.1,4.3);
			\node[below right] at (0.1,-0.3) {\footnotesize $x=\epsilon$};
			
			% \lambda_1=\lambda_2
			\draw[thick, blue] (0,0) -- (2,2);
			\node[gray,above right] at (1.3,1) {\footnotesize $x=y$};

			% standard deformation
			\draw[thick, red, samples=80, domain=0:2] 
			plot(\x, {-0.16155*\x^4+1.021216*\x^3-1.896216*\x^2+4});
			
			% notations
			\foreach \x in {2} {\draw (\x,0.05) -- (\x,-0.05) node[below] {\x};}
			\foreach \y in {2,4} {\draw (0.05,\y) -- (-0.05,\y) node[left] {\y};}
			
			% (2,2) standard metric
			\fill[yellow] (2,2) circle (1.2pt);
			\node[right] at (2,2) {\footnotesize standard metric};
		\end{tikzpicture}
	\end{minipage}
	{\caption{An illustration of the degree argument.}}\label{fig_3}	 
\end{figure}

\section{Refined relation for the first two eigenvalues}
\label{sec:first-two-refined}

According to Hersch's inequality \eqref{Hersch}, the first two eigenvalues satisfy the  inequality $\frac 1{\lambda_1}+\frac 2{\lambda_2} \ge \frac 32$. In what follows we prove a stronger  relation between $\lambda_1$ and $\lambda_2$. As a result, we are able to exclude the region shaded by pink lines in Figure \ref{fig:joint-region-intro} from the image of $\mathscr{R}$.

\begin{theorem}\label{thm:uniform-directional-bound}
Let $\mu$ be an admissible measure on $(\mathbb S^2,[g_{\rm rd}])$ with
$\mu(\mathbb S^2)=4\pi$. %For $a>0$, set
%\[ I_a=[0,\min\{a,3/2\}).\]
Then
\begin{equation}\label{eq:uniform-directional-bound}
 \lambda_2(\mu)^{-1}\ge
 \inf_{t\in [0,\lambda_1(\mu)^{-1})}
 \Theta_{\lambda_1^{-1}(\mu)}(t),
\end{equation}
where $\Theta_a$ is defined in \eqref{theta}. Numerically, when
$\lambda_1(\mu)=1$, the explicit one-dimensional bound gives
$\lambda_2(\mu)\lesssim3.260142$.
\end{theorem} 

Throughout this section $\mu$ is Hersch-balanced, has mass $4\pi$, and the
notation of Section \ref{subsec:matrix-prelim} is in force. A trace-only argument discards the direction selected by the first eigenfunction.  Retaining that direction gives a one-parameter inequality that can be solved for $1/\lambda_2$ for every prescribed value of
$\lambda_1$. We first need 
\begin{lemma} \label{lem:mx-loewner}
Assume $\lambda_1(\mu)<\lambda_2(\mu)$.  Let $\phi$ be a
Dirichlet-normalized first eigenfunction of $K_\mu$.  
Then
\begin{equation}\label{eq:mx-Bloewner}
	\frac{1}{\lambda_1(\mu)}P\phi \otimes P\phi
	\leq B
	\leq \frac{1}{\lambda_2(\mu)}I+\left(\frac{1}{\lambda_1(\mu)}-\frac{1}{\lambda_2(\mu)}\right) P\phi\otimes P\phi,
\end{equation}
and
\begin{equation}\label{eq:mx-Cloewner}
	C^*C
	\leq \frac{1}{\lambda_2(\mu)}B+\frac{1}{\lambda_1(\mu)}\left(\frac{1}{\lambda_1(\mu)}-\frac{1}{\lambda_2(\mu)}\right) P\phi\otimes P\phi-B^2.
\end{equation}
\end{lemma}

\begin{proof}
Let $\Pi$ be the
orthogonal projection onto $\mathbb{R}\phi$, and let $L= K_\mu- r_1\Pi$. By the spectral theorem,
\[ 
 L\Pi=\Pi L=0,
 \qquad 0\le L\le r_2(I-\Pi).
\]
Compressing the inequalities
$r_1\Pi\le K_\mu\le r_2I+(r_1-r_2)\Pi$ to $E$ gives
\eqref{eq:mx-Bloewner}, because $P\Pi P= P\phi \otimes P\phi$.  Moreover,
\[
 K_\mu^2=r_1^2\Pi+L^2
 \le r_1^2\Pi+r_2L
 =r_2K_\mu+r_1(r_1-r_2)\Pi.
\]
Compress this inequality to $E$ and use
$PK_\mu^2P=B^2+C^*C$ to obtain \eqref{eq:mx-Cloewner}.
\end{proof}

\begin{proposition}\label{B_1}
Assume $\lambda_2(\mu)>2$. Then $P\phi\ne0$. Use the notation of Lemma~\ref{lem:mx-loewner} and rotate the Dirichlet-orthonormal basis of $E$ so that $e_1\in\mathbb{R}P\phi$.  Then 
\begin{equation}\label{eq:mx-p-range}
 	\frac32-\frac2{\lambda_2(\mu)}\leq B_{11}\leq\frac1{\lambda_1(\mu)},\quad \frac{B_{22}+B_{33}}{2}\leq\frac{1}{\lambda_2(\mu)}.
\end{equation}
Moreover, one has 
\begin{equation}\label{B_2}
	B_{11}-\frac{1}{2}\leq2\sqrt{\frac{2}{7}B_{11}\left(\frac{1}{\lambda_1}-B_{11}\right)}
		+\frac{8}{\sqrt{21}}
		\sqrt{
			\frac{B_{22}+B_{33}}{2}
			\left(\frac{1}{\lambda_2}-\frac{B_{22}+B_{33}}{2}\right)
		}
\end{equation}
\end{proposition}

\begin{proof} 
First, $P\phi\ne0$. Indeed, if $P\phi=0$, the upper bound in
\eqref{eq:mx-Bloewner} gives $B\le\lambda_2^{-1}I$, and hence
$\operatorname{tr}B\le3/\lambda_2<3/2$, contradicting
$\operatorname{tr}B=3/2$. We may therefore rotate the basis as stated.

The upper bound in \eqref{eq:mx-Bloewner} gives
$B_{22},B_{33}\le1/\lambda_2(\mu)$. Since $\operatorname{tr}B=3/2$,
\[
 B_{11}\ge\frac32-\frac2{\lambda_2(\mu)},
 \qquad
 \frac{B_{22}+B_{33}}2\le\frac1{\lambda_2(\mu)}.
\]
Also $B\le\lambda_1^{-1}I$, so $B_{11}\le1/\lambda_1(\mu)$. This proves
\eqref{eq:mx-p-range}.

Let
\[
 R:=\frac1{\lambda_2}B
 +\frac1{\lambda_1}\left(\frac1{\lambda_1}-\frac1{\lambda_2}\right)
 P\phi\otimes P\phi-B^2.
\]
By \eqref{eq:mx-Cloewner}, $C^*C\le R$. Since
$B_{11}\ge\|P\phi\|_D^2/\lambda_1$ and $(B^2)_{11}\ge B_{11}^2$,
\begin{equation}\label{eq:mx-R11}
 R_{11}\le B_{11}\left(\frac1{\lambda_1}-B_{11}\right).
\end{equation}
For $j=2,3$,
\begin{equation}\label{eq:mx-Rperp}
 R_{jj}\le B_{jj}\left(\frac1{\lambda_2}-B_{jj}\right).
\end{equation}

Take $H=\operatorname{diag}(2/3,-1/3,-1/3)$. Then
$\langle A,H\rangle_{\rm HS}=B_{11}-1/2$, and the componentwise estimate
\eqref{(AH)_1} gives (Estimate \eqref{(AH)_1} is rotationally covariant and hence remains valid after the present change of orthonormal coordinates)
\begin{equation*}
\begin{split}
 B_{11}-\frac12
 &\le 2\sqrt{\frac27(C^*C)_{11}}
 +\frac 4{\sqrt{21}} \left(\sqrt{(C^*C)_{22}}+\sqrt{(C^*C)_{33}}\right)\\
 &\le 2\sqrt{\frac27B_{11}\left(\frac1{\lambda_1}-B_{11}\right)}\\
 &\quad+\frac 4{\sqrt{21}}
 \left[\sqrt{B_{22}\left(\frac1{\lambda_2}-B_{22}\right)}
 +\sqrt{B_{33}\left(\frac1{\lambda_2}-B_{33}\right)}\right]\\
 &\le 2\sqrt{\frac27B_{11}\left(\frac1{\lambda_1}-B_{11}\right)}
 +\frac 8{\sqrt{21}}
 \sqrt{\frac{B_{22}+B_{33}}2
 \left(\frac1{\lambda_2}-\frac{B_{22}+B_{33}}2\right)},
\end{split}
\end{equation*}
which is \eqref{B_2}. 
\end{proof}

\begin{proof}[Proof of Theorem \ref{thm:uniform-directional-bound}]
First assume $\lambda_2(\mu)>2$. From \eqref{B_2}, we have
\begin{equation*}
 \frac1{\lambda_2}\ge
 \frac{B_{22}+B_{33}}2+
 \frac{21}{32(B_{22}+B_{33})}
 \left(B_{11}-\frac12-
 2\sqrt{\frac27B_{11}(\frac 1{\lambda_1}-B_{11})}\right)_+^2.
\end{equation*}
Let $t=B_{11}$, one has
$t\in [0,\frac 1{\lambda_1})$. Define, for any $a \ge \frac{1}{2}$ and  $t\in [0,a)$,
\begin{equation}\label{theta}
 \Theta_a(t)=
 \frac{21}{32(\frac32-t)}
 \left(t-\frac12-2\sqrt{\frac27t(a-t)}\right)_+^2
 +\frac{\frac32-t}{2}.
\end{equation}
Since $B_{22}+B_{33}=\frac 32 - B_{11}\geq 0$, the preceding inequality gives
$\lambda_2(\mu)^{-1}\ge\Theta_{\lambda_1(\mu)^{-1}}(B_{11})$, and taking the infimum over
$t \in [0, \lambda_1(\mu)^{-1})$ proves the theorem.

If $\lambda_2(\mu)\le2$, then $\lambda_2(\mu)^{-1}\ge\frac{1}{2}$. Hersch's bound $\lambda_1\le2$ gives
$\lambda_1(\mu)^{-1}\ge\frac{1}{2}$, hence 
\[ \inf_{[0,\lambda_1(\mu)^{-1})}\Theta_{\lambda_1(\mu)^{-1}}(t)\leq\Theta_{\lambda_1(\mu)^{-1}}(\frac{1}{2})=\frac{1}{2}\leq\lambda_2(\mu)^{-1}.\]
Thus the claimed inequality is proved.
\end{proof}

Here are some numerical values of the explicit bound. For $0<\lambda_1\le2$, define $\Xi(\lambda_1)=\inf_{t\in[0,\lambda_1^{-1})}\Theta_{\lambda_1^{-1}}(t)$.

\begin{table}[htbp]
	\begin{equation*}
		\begin{array}{cc|cc|cc}
			\lambda_1&\Xi(\lambda_1)^{-1}
			&\lambda_1&\Xi(\lambda_1)^{-1}
			&\lambda_1&\Xi(\lambda_1)^{-1}\\ \hline
			0.85&3.973881&1.25&2.668871&1.65&2.218477\\
			0.90&3.681223&1.30&2.591283&1.70&2.180762\\
			0.95&3.449119&1.35&2.521836&1.75&2.145608\\
			1.00&3.260142&1.40&2.459220&1.80&2.112743\\
			1.05&3.102973&1.45&2.402395&1.85&2.081935\\
			1.10&2.969945&1.50&2.350527&1.90&2.052986\\
			1.15&2.855675&1.55&2.302939&1.95&2.025724\\
			1.20&2.756277&1.60&2.259075&2.00&2.000000\\
			\hline
		\end{array}
	\end{equation*}
	\caption{Numerical results of Theorem \ref{thm:uniform-directional-bound}.}
\end{table}
Numerically, if $\lambda_1=0.85$, the bound gives $\lambda_2\lesssim3.973881$; for $\lambda_1=0.86$ it gives $\lambda_2\lesssim3.909361$. These values, when compared with Table~\ref{Numerical result of improved nad}, give the approximate location of the intersection of the green curve and the red curve in Figure \ref{fig:joint-region-intro}. %The numerical crossover with the horizontal value $4$ occurs near $\lambda_1\approx0.8460963$.

\section*{Declaration on the Use of Artificial Intelligence}

During the preparation of this manuscript, the authors used ChatGPT for language polishing and grammar checking. ChatGPT was  used to assist in generating the three numerical tables, which were independently verified by the authors via MATLAB.  ChatGPT was also used to assist with the heavy polynomial computations in Lemma 3.2, Lemma 3.4 and Lemma 4.3. The authors independently verified all computations, reviewed and edited all AI-assisted content, and take full responsibility for the content of the manuscript.

%\bibliographystyle{alpha}
%\bibliography{ref}

%\iffalse 

%\fi
\end{document}